\documentclass{amsart}

\usepackage[a4paper,hmargin=3.70cm,vmargin=4cm]{geometry}
\usepackage{amsfonts,amssymb,amscd,amstext}
\usepackage{graphicx}
\usepackage[dvips]{epsfig}
\usepackage{hyperref}

\renewcommand{\geq}{\geqslant}

\renewcommand{\leq}{\leqslant}

\newcommand{\ptl}{\partial}
\newcommand{\Sg}{\Sigma} 
\newcommand{\sg}{\sigma}
\newcommand{\Om}{\Omega}
\newcommand{\eps}{\varepsilon}
\newcommand{\de}{\mathcal{D}}
\newcommand{\var}{\varphi}

\newcommand{\subeq}{\subseteq}
\newcommand{\sub}{\subset}
\newcommand{\rr}{\mathbb{R}}
\newcommand{\rrn}{\mathbb{R}^{n+1}}
\newcommand{\sph}{\mathbb{S}}

\newcommand{\escpr}[1]{\left< #1\right>}
\newcommand{\cone}{{\times\!\!\!\!\times}}

\DeclareMathOperator{\divv}{div}

\newtheorem{theorem}{Theorem}[section]
\newtheorem{proposition}[theorem]{Proposition}
\newtheorem{lemma}[theorem]{Lemma}
\newtheorem{corollary}[theorem]{Corollary}

\theoremstyle{definition}

\newtheorem{remark}[theorem]{Remark}
\newtheorem{remarks}[theorem]{Remarks}

\newtheorem{example}[theorem]{Example}

\newtheorem{examples}[theorem]{Examples}

\numberwithin{equation}{section}

\begin{document}

\title[Stable CMC hypersurfaces for homogeneous densities in convex cones]{Compact stable hypersurfaces with free boundary in convex solid cones with homogeneous densities}

\author[A.~Ca\~nete]{Antonio Ca\~nete}
\address{Departamento de
Matem\'atica Aplicada I \\
Universidad de Sevilla \\ E-41012 Sevilla, Spain}
\email{antonioc@us.es}

\author[C.~Rosales]{C\'esar Rosales}
\address{Departamento de
Geometr\'{\i}a y Topolog\'{\i}a \\
Universidad de Granada \\ E-18071 Granada, Spain}
\email{crosales@ugr.es}

\thanks{Both authors are partially supported by MCyT
research project MTM2010-21206-C02-01, and Junta de Andaluc\'ia grants FQM-325 and P09-FQM-5088}
\date{\today}

\begin{abstract}
We consider a smooth Euclidean solid cone endowed with a smooth homogeneous density function used to weight Euclidean volume and hypersurface area. By assuming convexity of the cone and a curvature-dimension condition, we prove that the unique compact, orientable, second order minima of the weighted area under variations preserving the weighted volume and with free boundary in the boundary of the cone are intersections with the cone of round spheres centered at the vertex. 
\end{abstract}

\maketitle

\thispagestyle{empty}

\section{Introduction}
A \emph{stable hypersurface} in a Riemannian manifold with boundary is a second order minimum of the area relative to the interior of the manifold for compactly supported deformations preserving the separated volume. From the first variation formulae such a hypersurface has constant mean curvature and free boundary meeting orthogonally the boundary of the manifold. The second variation formula implies that the index form associated to the hypersurface is nonnegative for smooth mean zero functions with compact support. The stability condition has been intensively studied in the literature and plays a central role in relation to the \emph{isoperimetric problem}, where we seek sets of the least possible perimeter among those with fixed volume.

Barbosa and do Carmo in \cite{bdc}, and Barbosa, do Carmo and Eschenburg in \cite{bdce} used suitable deformations to show that a smooth, compact, orientable and stable hypersurface $\Sg$ immersed in a Riemannian space form is a geodesic sphere. In \cite{wente}, Wente observed that the variation employed in \cite{bdc} is geometrically obtained by parallel hypersurfaces dilated to keep the enclosed volume constant. The resulting variation strictly decreases the boundary area unless the hypersurface coincides with a round sphere. This technique was successfully used later by Morgan  and Ritor\'e \cite{morgan-rit} to prove that geodesic spheres about the vertex and boundaries of flat round balls are the unique compact stable hypersurfaces inside smooth spherical cones of non-negative Ricci curvature and empty boundary. In smooth, convex, Euclidean solid cones with non-empty boundary, Ritor\'e and the second author showed in \cite{cones} that the unique compact stable hypersurfaces are either spherical caps centered at the vertex or half-spheres lying in a flat portion of the boundary of the cone. Also by following the ideas in \cite{bdc} and \cite{wente}, the Wulff shapes were characterized as the unique compact stable hypersurfaces with constant anisotropic mean curvature in $\rrn$, see the papers by Palmer \cite{palmer} and Winklmann \cite{winklmann}. 

The study of variational problems related to the minimization of the area functional in \emph{metric measure spaces} has been focus of attention in the last years, with an increasing development of the theory of minimal and constant mean curvature surfaces in this setting. In this paper we obtain new classification results for compact stable hypersurfaces in the framework of \emph{manifolds with density}. These structures have been considered by many authors and are currently a topic of much interest. They have applications in several areas ranging from functional analysis and probability theory to Riemannian geometry. In fact, many important notions in Riemannian geometry have generalizations to manifolds with density, allowing the extension of some classical questions and results. For a nice introduction to manifolds with density we refer the reader to Chapter~8 of Morgan's book \cite{gmt} and to Chapter~3 of Bayle's thesis \cite{bayle-thesis}. 

Let us introduce the definition of a manifold with density and the corresponding notions of volume, area and Ricci curvature. By a \emph{manifold with density} we mean a connected manifold $M^{n+1}$ with a Riemannian metric $\escpr{\cdot\,,\cdot}$ and a smooth $(C^\infty)$ positive function $f=e^\psi$ used to weight the Hausdorff measures associated to the Riemannian distance. In particular, the \emph{weighted volume} of a Borel set $\Om\subeq M$ and the \emph{weighted area} of a piecewise smooth hypersurface $\Sg\sub M$ relative to an open subset $U\subeq M$ are given by
\begin{equation}
\label{eq:volarea}
V_f(\Om):=\int_\Om dv_f=\int_\Om f\,dv,\qquad 
A_f(\Sg,U):=\int_{\Sg\cap U} da_f=\int_{\Sg\cap U} f\,da,
\end{equation}
where $dv$ and $da$ are the Riemannian elements of volume and area, respectively.  We also denote $dl_f:=f\,dl$, where $dl$ is the $(n-1)$-dimensional Hausdorff measure in $M$. Manifolds with density are also known as \emph{smooth metric measure spaces}. In fact, if we consider the Riemannian distance in $M$ and the weighted volume defined above, then we obtain a metric measure space whose associated Minkowski content coincides with the weighted area of the boundary for Borel sets with Lipschitz boundary. For a manifold $M$ with density $f=e^\psi$ the \emph{Bakry-\'Emery-Ricci tensor}, or simply \emph{$f$-Ricci tensor}, is defined by
\begin{equation}
\label{eq:fricci}
\text{Ric}_f:=\text{Ric}-\nabla^2\psi,
\end{equation}
where $\text{Ric}$ and $\nabla^2$ denote the Ricci tensor and the Hessian operator for the Riemannian metric in $M$, respectively.  We also consider the \emph{$k$-dimensional Bakry-\'Emery-Ricci tensor} given by
\begin{equation}
\label{eq:fkricci}
\text{Ric}^k_f:=\text{Ric}_f-\frac{1}{k}\,(d\psi\otimes d\psi), \ \text{ for any } k\neq 0.
\end{equation} 
The tensor $\text{Ric}_f$ was first introduced by Lichnerowicz \cite{lich1}, \cite{lich2}, and later generalized in a form equivalent to $\text{Ric}^k_f$ by Bakry and \'Emery \cite{be} in the framework of diffusion generators. Note that $\text{Ric}_f^k$ carries information involving curvature and dimension from both, the Riemannian manifold $M$ and the density $f=e^\psi$. By this reason, a lower bound on $\text{Ric}^k_f$ is usually known as a \emph{curvature-dimension condition}. Such a condition allows the extension to manifolds with density of many classical comparison results in Riemannian geometry, see \cite{qian}, \cite{lott},  \cite{morgandensity}, \cite{morganmyers},  \cite{wwsurvey}, \cite{wei}, \cite{rimoldi} and references therein. Moreover, these Ricci tensors appear explicitly in the second derivative of the weighted area, see \eqref{eq:2nd} and \eqref{eq:index1}, and are useful to obtain extensions of the L\'evy-Gromov isoperimetric inequality, see \cite[Ch.~3]{bayle-thesis} and \cite{milman}. We also remark that the equation $\text{Ric}_f=\lambda\escpr{\cdot\,,\cdot}$ for some constant $\lambda$ is the gradient Ricci soliton equation, which plays an important role in the theory of the Ricci flow.

Once we have notions of volume and area, we can study minimization problems such as the \emph{isoperimetric problem} or the \emph{Plateau problem}, where we try to find hypersurfaces of the least possible area with a boundary or volume constraint. The complete solution to these questions inside an arbitrary manifold with density is a very difficult task. Therefore, it is natural to focus first on the most simple and symmetric Riemannian spaces -the space forms- endowed with some suitable densities. Here by a suitable density we mean one with a simple behaviour of the Ricci tensors (constant, bounded) or having good properties with respect to a certain subgroup of diffeomorphisms\,/\,isometries of the ambient manifold. Probably the best studied example is the Euclidean space $\rrn$ endowed with a radial density. This includes the Gaussian density $\exp(-|p|^2)$, the model density $\exp(|p|^2)$, and the homogeneous densities $|p|^k$ with $k\in\rr$. Though there are several works studying constant mean curvature, stable, and isoperimetric hypersurfaces for these densities, we do not aim here to give an exhaustive list of references. However, we would like to point out that the \emph{log-convex density conjecture}, posed by Bayle, Morgan and the authors in \cite{rcbm}, is a remarkable open question in this setting. It was shown in \cite{rcbm} that a radial density is log-convex if and only if round spheres centered at the origin are stable hypersurfaces. Hence, it is natural to ask if these spheres are also global minimizers of the area under a volume constraint. The existence of isoperimetric solutions for radial log-convex densities is a consequence of more general results for radial densities proved by Morgan and Pratelli \cite{morgan-pratelli}. Figalli and Maggi have recently shown in \cite{figalli} that the log-convex density conjecture holds for some interval of volumes $[0,m_0)$. In \cite{howe}, Howe states that round spheres about the origin are isoperimetric in $\rrn-\{0\}$ with any radial density such that the weighted area of such spheres is a convex function of the weighted bounded volume and satisfying some additional minor hypotheses. A complete account of all the related results until 2011 has been given by Morgan in \cite{lcdc}. 

In this paper we focus on densities defined on \emph{solid cones} of $\rrn$ and having a nice behaviour with respect to the family of dilations centered at the origin. The motivation for this comes from the existence of many important results in the theory of constant mean curvature and stable hypersurfaces in $\rrn$ whose proofs rely on the scaling property that some geometric quantities (volume, area and mean curvature) show under dilations.  Given a smooth solid cone $M\subeq\rrn$ and a smooth density $f=e^\psi$ defined on the punctured cone $M^*$, we consider the weighted volume $V_f$ and the weighted area $A_f$ in \eqref{eq:volarea} \emph{relative to the interior of the cone}. This means that the intersection $\Sg\cap\ptl M$ of a hypersurface with the boundary of the cone does not contribute to $A_f(\Sg)$. Then, we prove in Lemma~\ref{lem:vap} that $V_f$ and $A_f$ are homogeneous with respect to the one-parameter group of dilations $h_t(p):=tp$ if and only if $f$ is a homogeneous function. This leads us to consider $k$-\emph{homogeneous densities}: these are positive functions whose restriction to any open segment of the cone leaving from the origin is a monomial of degree $k$. A basic example is the Euclidean space $\rrn$ with constant density $f=1$.  As it is observed in Remark~\ref{re:singularity} a homogeneous density of degree $k\neq 0$ cannot be extended to the origin as a continuous and positive function. Hence the origin may be understood as a \emph{singularity} for such densities and must be treated carefully. In fact, for densities of degree $k\leq-n$ we see in Examples~\ref{ex:infinite} that any smooth, compact, embedded hypersurface containing the origin has infinite weighted area. So, in order to study minimizers of the area functional it is natural, in the case $k<0$, to restrict ourselves to hypersurfaces contained in the punctured cone.

In Section~\ref{sec:homogeneous} we provide several examples and properties of homogeneous densities, including two characterizations of the curvature-dimension condition $\text{Ric}_f^k\geq 0$. An important result is given in Proposition~\ref{prop:keyfacts}, where we show that the position vector field $X(p):=p$ has constant \emph{$f$-divergence}, and that the \emph{$f$-mean curvature} of hypersurfaces scales with respect to dilations as the Euclidean mean curvature. The \emph{$f$-divergence} of a smooth vector field is defined in \eqref{eq:divvol} and \eqref{eq:divsup}. It provides a generalization of the Riemannian divergence functional which allows extending to manifolds with density some classical divergence theorems and integration by parts formulae in Riemannian geometry, see Section~\ref{sec:divth}. The \emph{$f$-mean curvature} of a hypersurface is the function defined in \eqref{eq:fmc} and previously introduced by Gromov \cite{gromov-GAFA} in relation to the first derivative of the weighted area functional, see \eqref{eq:1st}.

The results of this paper are closely related to the isoperimetric problem in Euclidean solid cones with homogeneous densities. In the classical case of $\rrn$ with constant density $f=1$, it was proved by Lions and Pacella \cite{lp} that, inside a convex cone, round balls centered at the vertex minimize the relative perimeter among regions enclosing the same volume. In the same paper they obtain uniqueness of the isoperimetric regions for smooth convex cones. The uniqueness in the smooth case is also consequence of the classification of stable hypersurfaces given by Ritor\'e and the second author in \cite{cones}. Recently Figalli and Indrei \cite{figalli-indrei} have discussed uniqueness in arbitrary convex cones. 

The isoperimetric question has also been studied in $\rrn$ with the homogeneous radial density $|p|^k$. For $k<-(n+1)$, it was first shown by Carroll, Jacob, Quinn and Walters \cite{cjqw} for the planar case and later by D\'iaz, Harman, Howe and Thompson \cite{dhht} for arbitrary dimension, that round spheres centered at the origin are the unique isoperimetric hypersurfaces (bounding volume away from the origin). For $n=1$ and $k>0$ it was proved by Dahlberg, Dubbs, Newkirk and Tran \cite{dahlberg} that round circles passing through the origin are the unique solutions. The description of the isoperimetric curves in planar cones with density $|p|^k$, $k>0$, has been given in \cite{dhht}. 

In \cite{cabre2}, Cabr\'e and Ros-Oton establish that, in $\rrn$ with a monomial homogenous density $|x_1|^{\alpha_1}\cdots |x_{n+1}|^{\alpha_{n+1}}$, where $\alpha_i\geq 0$ for any $i=1,\ldots, n+1$, the intersections with the cone $\{(x_1,\ldots, x_{n+1})\in\rrn\,; x_i\geq 0\text{ for all } i \text{ such that }\alpha_i>0\}$ of round spheres centered at the origin are isoperimetric solutions. Very recently  Cabr\'e, Ros-Oton and Serra in \cite{cabre} and \cite{cabre3} have applied the ABP method to a linear Neumann problem with density to obtain the following nice result: the spherical caps centered at the vertex are isoperimetric solutions for an arbitrary convex cone endowed with a continuous, homogeneous, non-negative density function $f$ of degree $k>0$, such that $f$ is positive and locally Lipschitz in the interior of the cone, and $f^{1/k}$ is concave in the interior of the cone. This was proved in \cite{cabre} under the further assumptions that $f$ is $C^{1,\alpha}$ with $0<\alpha<1$ in the interior of the cone and $f$ vanishes along the boundary. The general case has been treated in \cite{cabre3}, together with an extension for anisotropic area functionals associated to homogeneous densities. In particular, this generalizes the classical result of Lions and Pacella \cite{lp}, the Wulff isoperimetric inequality proved by Taylor \cite{taylor}, and a previous result of Maderna and Salsa \cite{maderna} for the half-plane $y>0$ with density $y^k$, $k>0$. Also, in the introduction of \cite{cabre3} it is announced that the uniqueness of the isoperimetric solutions will be discussed in a future work with Cinti and Pratelli. In \cite{homoiso} the authors obtain existence results and properties of the isoperimetric profile under mild regularity conditions on the homogeneous density, and uniqueness of the isoperimetric solutions for smooth convex cones with $\text{Ric}_f^k\geq 0$ and $k>0$. It is worth mentioning that, as a consequence of Lemma~\ref{lem:matrimonio}, the curvature-dimension inequa\-lity $\text{Ric}_f^k\geq 0$ for $k>0$ is equivalent to the concavity of $f^{1/k}$, which is the key hypothesis assumed in the main results of \cite{cabre} and \cite{cabre3}. 

Our main goal in this paper is to prove classification results for compact stable hypersurfaces inside Euclidean solid cones with homogeneous densities. An \emph{$f$-stable hypersurface} in a Riemannian manifold $M$ with a smooth density $f=e^\psi$ satisfies that the second derivative of the weighted area functional is nonnegative under compactly supported variations preserving the enclosed volume. Hence an isoperimetric hypersurface is also an $f$-stable one. Recently, many authors have studied complete $f$-stable minimal surfaces inside $3$-manifolds with non-negative $f$-Ricci tensor or $f$-scalar curvature, see \cite{fan}, \cite{ho}, \cite{mejia}, \cite{espinar}, \cite{liu}, \cite{impera} and \cite{kathe}. The $f$-stability of hyperplanes for Euclidean product densities has been considered in \cite{chiara}, see also \cite{calibrations}. The variational properties of critical points of the weighted area for piecewise regular densities have been established in \cite[Sect.~2.2]{cvm}.

In Section~\ref{sec:variational} we gather some variational properties and examples of $f$-stable hypersurfaces \emph{with free boundary} in a Euclidean solid cone $M\subeq\rrn$ endowed with a $k$-homogenous density $f=e^\psi$. The variation formulae for hypersurfaces with non-empty boundary obtained in \cite{kathe} imply that $f$-stable hypersurfaces have constant $f$-mean curvature off of the vertex and meet $\ptl M$ orthogonally along the free boundary.  Moreover, the associated \emph{$f$-index form} $\mathcal{Q}_f$ defined in \eqref{eq:index2} satisfies $\mathcal{Q}_f(u,u)\geq 0$ for any smooth function $u$ supported away from the origin and having mean zero with respect to the weighted element of area. In Example~\ref{ex:spheres11} we show that the spherical caps centered at the vertex always provide critical points of the area under a volume constraint for arbitrary homogeneous densities. Hence they are natural candidates to be $f$-stable. In fact, we see in Example~\ref{ex:sphere2} that, for $k\leq -n$, the spherical caps centered at the vertex are always $f$-stable. This also happens when $k>0$ and $\text{Ric}^k_f\geq 0$ as a consequence of the isoperimetric results in \cite{cabre3} and \cite{homoiso}. Then, it is natural to ask under which conditions the spherical caps centered at the vertex are the unique $f$-stable hypersurfaces in a Euclidean solid cone with a homogeneous density. To find such a condition it is interesting to observe that the $f$-index form depends on the $f$-Ricci tensor $\text{Ric}_f$ and the second fundamental form $\text{II}$ of $\ptl M$ in such a way that the stability inequality $\mathcal{Q}_f(u,u)\geq 0$ is more restrictive provided $\text{Ric}_f\geq 0$ and $\text{II}\geq 0$. Hence convexity of the cone $M$ and nonnegativity of $\text{Ric}_f$ become natural hypotheses in order to obtain sharp classification results for $f$-stable hypersurfaces. In fact, the main result of the paper establishes the following:
\begin{quotation}
\emph{Let $\Sg$ be a smooth, compact, orientable hypersurface in a convex solid cone $M\subeq\rrn$ endowed with a $k$-homogeneous density of nonnegative Bakry-\'Emery-Ricci tensor $\emph{Ric}_f^k$. Suppose that, either $k<-n$ and $\Sg\sub M^*$, or $k>0$ and $\Sg\sub M$. If $\Sg$ is $f$-stable, then $\Sg$ is the intersection with the cone of a round sphere centered at the vertex.}
\end{quotation}

The proof of the previous theorem is contained in Section~\ref{sec:main} and it is divided into two steps. We first obtain in Theorem~\ref{th:main} that the statement holds for an $f$-stable hypersurface $\Sg$ contained in the punctured cone $M^*$. For this, we insert inside the stability inequality $\mathcal{Q}_f(u,u)\geq 0$ the test function given by
\[
u=n+k+H_f\escpr{X,N},
\]
where $H_f$ is the $f$-mean curvature of the hypersurface, $X(p)=p$ is the position vector field, and $N$ is a unit normal along $\Sg$. In $\rrn$ with constant density $f=1$ the function $u$ coincides, up to a constant, with the function used by Barbosa and do Carmo in \cite{bdc}. As an application of the divergence theorem in Lemma~\ref{lem:divthsup} we deduce that $u$ has mean zero with respect to the weighted element of area. In fact, this integral equality leads us in Proposition~\ref{prop:mink} to generalize to solid cones with homogeneous densities a classical Minkowski identity for compact hypersurfaces in $\rrn$ relating volume, area and mean curvature \cite{hsiung}. From a geometric point of view, we show in Lemma~\ref{lem:inter} that the test function $u$ is associated to the deformation of $\Sg$ obtained when one leaves by equidistant hypersurfaces and then applies a dilation centered at the vertex to restore the separated volume. This variation decreases the weighted area unless $\Sg$ satisfies equality in \eqref{eq:sign}, which indicates in some weighted sense that $\Sg$ is totally umbilical, compare with \cite[Lem.~3.2]{bdc}. By using Lemma~\ref{lem:umbilical} we conclude that $\Sg$ is the intersection with the cone of a round sphere centered at the vertex. Second, we prove in Theorem~\ref{th:main2} that the statement holds for homogeneous densities of degree $k>0$ and $f$-stable hypersurfaces \emph{that may contain the vertex of the cone}. In fact, an approximation argument similar to the one in \cite[Sect.~3]{morgan-rit} or \cite[Sect.~4]{cones} shows that the proof of Theorem~\ref{th:main} can be carried out even in the case $0\in\Sg$. It is worth pointing out that the hypothesis $k>0$ is essential to perform this approximation scheme since it ensures that the density is bounded near the singularity. We finish Section~\ref{sec:main} with a list of interesting examples and remarks showing the sharpness of our result.

Finally, in Section~\ref{sec:main2} we characterize \emph{strongly $f$-stable hypersurfaces} in solid cones with 
$k$-homogeneous densities. Such hypersurfaces are critical points of the area under a volume constraint, and with non-negative second derivative of the area \emph{for any compactly supported variation}. In Example~\ref{ex:sphere2} we see that the intersection with the cone of any round sphere centered at the vertex is strongly $f$-stable if and only if $k\leq -n$. In fact, as a consequence of Theorem~\ref{th:main} we deduce that these are the unique compact strongly $f$-stable hypersurfaces in the punctured cone when $k<-n$, the cone is convex and $\text{Ric}^k_f\geq 0$. Observe that the case $k=-n$ is not covered in Theorem~\ref{th:main}. This case is particularly interesting since, as an application of the Minkowski formula \eqref{eq:mink2}, any critical point of the weighted area under a volume constraint has vanishing $f$-mean curvature. In Theorem~\ref{th:minimal} we show that, assuming convexity of the cone and the curvature-dimension condition $\text{Ric}^{-n}_f\geq 0$, round spheres centered at the vertex and intersected with the cone are the unique compact strongly $f$-stable hypersurfaces contained in the punctured cone. 

The paper is organized as follows. In Section~\ref{sec:divth} we obtain extensions to manifolds with density of some classical divergence theorems and integration by parts formulae in Riemannian geometry. They provide basic tools that will be used extensively throughout the paper. In Section~\ref{sec:homogeneous} we introduce the family of homogeneous densities in Euclidean solid cones and study their main properties. In Section~\ref{sec:variational} we gather some variational features, with emphasis in the characterization of $f$-stationary and $f$-stable hypersurfaces. Finally Sections~\ref{sec:main} and \ref{sec:main2} contain our classification results for compact $f$-stable and strongly $f$-stable hypersurfaces, respectively.

\section{Divergence theorems in manifolds with density}
\label{sec:divth}
\setcounter{equation}{0}

In this section we introduce suitable notions of divergence for vector fields in manifolds with density. Then we generalize some classical divergence theorems and integration by parts formulae in Riemannian geometry.

Let $M$ be an $(n+1)$-dimensional smooth and oriented Riemannian manifold with boundary $\ptl M$. The Riemannian divergence operator $\divv$ acting on vector fields is the adjoint, with respect to the Riemannian volume $dv$, to the gradient operator $-\nabla$ acting on functions. In fact, for any smooth vector field $X$ in $M$ and any function $\varphi\in C^\infty_0(M)$ vanishing along $\ptl M$, the divergence theorem implies that
\[
\int_M\varphi\,\divv X\,dv=-\int_M\escpr{\nabla\varphi,X}dv,
\]
since $\divv(\varphi X)=\varphi\,\divv X+\escpr{\nabla\varphi,X}$. It follows that the Laplacian $\Delta\varphi:=\divv(\nabla\varphi)$ defines a self-adjoint operator, in the sense that
\[
\int_M \varphi_1\,\Delta\varphi_2\,dv=\int_M \varphi_2\,\Delta\varphi_1\,dv=-\int_M\escpr{\nabla\varphi_1,\nabla\varphi_2}dv,
\]
for any two functions $\varphi_1,\varphi_2\in C^\infty_0(M)$ vanishing along $\ptl M$.

In order to obtain similar formulae in $M$ endowed with a density $f=e^\psi$ we define the $f$-\emph{divergence} of a smooth vector field $X$ by means of equality
\begin{equation}
\label{eq:divvol}
\divv_f X:=(1/f)\divv(fX)=\divv X+\escpr{\nabla\psi,X}.
\end{equation}
It is then easy to check that the operator $\divv_f$ is the adjoint to $-\nabla$ with respect to the weighted volume $dv_f$, that is
\[
\int_M\varphi\,\divv_f X\,dv_f=-\int_M\escpr{\nabla\varphi,X}dv_f,
\]
for any $\varphi\in C^\infty_0(M)$ vanishing along $\ptl M$. Therefore the $f$-\emph{Laplacian operator} given by 
\[
\Delta_f\varphi:=\divv_f(\nabla\varphi)=\Delta\varphi+\escpr{\nabla\psi,\nabla\varphi}
\] 
is self-adjoint with respect to $dv_f$.

As an immediate consequence of the Riemannian divergence theorem and the fact that $\divv_fX\,dv_f=\divv(fX)\,dv$ we can deduce the following result.

\begin{lemma}[Divergence theorem in manifolds with density]
\label{lem:divthvol}
Let $M$ be an oriented Riemannian manifold endowed with a density $f$. For any smooth vector field $X$ with compact support on $M$, and any open set $\Om\subeq M$ with piecewise smooth boundary, we have
\[
\int_\Om\divv_f X\,dv_f=-\int_{\ptl\Om}\escpr{X,N}da_f,
\]
where $N$ is the inner unit normal along $\ptl\Om$. 
\end{lemma}

Suppose now that $\Sg$ is an orientable smooth hypersurface in $M$ and let $N$ be a unit normal vector along $\Sg$. For any smooth vector field $X$ with compact support on $\Sg$ we denote $X^\bot=\escpr{X,N}N$ and $X^\top=X-X^\bot$. As an application of the Riemannian divergence theorem we get the well-known formula
\begin{equation}
\label{eq:wes}
\int_\Sg\divv_\Sg X\,da=-\int_\Sg nH\escpr{X,N}da-\int_{\ptl\Sg}\escpr{X,\nu}dl,
\end{equation}
where $\divv_\Sg$ is the Riemannian divergence relative to $\Sg$, $H=(-1/n)\divv_\Sg N$ is the Riemannian mean curvature, and $\nu$ is the inner unit normal along $\ptl\Sg$ in $\Sg$. In particular, if $X$ is tangent to $\Sg$ and $\varphi\in C^\infty_0(\Sg)$, then 
\[
\int_\Sg\varphi\,\divv_\Sg X\,da=-\int_\Sg\escpr{\nabla_\Sg\varphi,X}da-\int_{\ptl\Sg}\varphi\,\escpr{X,\nu}dl,
\]
where $\nabla_\Sg$ stands for the gradient operator relative to $\Sg$. This implies the classical integration by parts formula
\[
\int_\Sg\varphi_1\,\Delta_\Sg\varphi_2\,da=-\int_\Sg\escpr{\nabla_\Sg\varphi_1,\nabla_\Sg\varphi_2}da-\int_{\ptl\Sg}\varphi_1\,\frac{\ptl\varphi_2}{\ptl\nu}\,dl,
\]
for $\varphi_1,\varphi_2\in C^\infty_0(\Sg)$, where $\Delta_\Sg\varphi:=\divv_\Sg(\nabla_\Sg\varphi)$ is the Laplacian operator relative to $\Sg$ and $\ptl\varphi_2/\ptl\nu$ denotes the directional derivative of $\varphi_2$ with respect to $\nu$.

Given a density $f=e^\psi$ in $M$, we define the $f$-\emph{divergence relative} to $\Sg$ of $X$ as the function
\begin{equation}
\label{eq:divsup}
\divv_{\Sg,f}X:=\divv_\Sg X+\escpr{\nabla\psi,X}.
\end{equation} 
For any $\varphi\in C^\infty(\Sg)$ it is easy to check that
\begin{equation}
\label{eq:beq}
\divv_{\Sg,f}(\varphi X)=\varphi\,\divv_{\Sg,f}X+\escpr{\nabla_\Sg\varphi,X}.
\end{equation}
The $f$-\emph{mean curvature} of $\Sg$ with respect to $N$ is the function
\begin{equation}
\label{eq:fmc}
H_f:=-\divv_{\Sg,f} N=nH-\escpr{\nabla\psi,N}.
\end{equation}
The previous definition of $H_f$ coincides with the ones introduced in \cite{gromov-GAFA}, \cite[Chap.~3]{bayle-thesis}, \cite[Sect.~3]{rcbm}, and it is related to the first derivative of the weighted area functional, see \eqref{eq:1st}. 

We can now prove the following generalization of formula \eqref{eq:wes}.

\begin{lemma}[Divergence theorem for hypersurfaces in manifolds with density]
\label{lem:divthsup}
Let $\Sg$ be a smooth hypersurface with unit normal vector $N$ in an oriented Riemannian manifold endowed with a density function $f=e^\psi$. Then, for any smooth vector field $X$ with compact support on $\Sg$, we have
\[
\int_\Sg\divv_{\Sg,f}X\,da_f=-\int_\Sg H_f\escpr{X,N}da_f
-\int_{\ptl\Sg}\escpr{X,\nu}dl_f,
\]
where $H_f$ is the $f$-mean curvature defined in \eqref{eq:fmc} and $\nu$ is the inner unit normal along $\ptl\Sg$ in $\Sg$.
\end{lemma}

\begin{proof}
By using equality $\divv_\Sg(fX)=f\divv_\Sg X+\escpr{\nabla_\Sg f,X}$ and equation \eqref{eq:wes}, we get
\begin{align*}
\int_\Sg\divv_{\Sg,f}X\,da_f&=\int_\Sg f\,\big(\divv_\Sg X+\escpr{\nabla\psi,X}\big)\,da
\\
&=\int_\Sg\divv_\Sg(fX)\,da+\int_\Sg\escpr{\nabla f-\nabla_\Sg f,X}da
\\
&=-\int_\Sg nHf\escpr{X,N}da-\int_{\ptl\Sg}f\escpr{X,\nu}dl+\int_\Sg f\escpr{\nabla\psi,N}\escpr{X,N}da.
\\
&=-\int_\Sg H_f\escpr{X,N}da_f-\int_{\ptl\Sg}\escpr{X,\nu}dl_f,
\end{align*}
and the proof follows.
\end{proof}

Finally, we define the $f$-\emph{Laplacian relative} to $\Sg$ of a function $\varphi\in C^\infty(\Sg)$ as the second order linear operator
\begin{equation}
\label{eq:fsigmalaplacian}
\Delta_{\Sg,f\,}\varphi:=\divv_{\Sg,f}(\nabla_\Sg\varphi)=\Delta_\Sg\varphi+\escpr{\nabla_\Sg\psi,\nabla_\Sg\varphi}.
\end{equation}
As an immediate consequence of Lemma~\ref{lem:divthsup} and equation \eqref{eq:beq} we get this result.

\begin{corollary}[Integration by parts for hypersurfaces in manifolds with density]
\label{cor:ibp}
Let $\Sg$ be a smooth orientable hypersurface in an oriented Riemannian manifold with a density function $f=e^\psi$. Then, for any two functions $\varphi_1,\varphi_2\in C^\infty_0(\Sg)$, we have
\[
\int_\Sg\varphi_1\,\Delta_{\Sg,f}\,\varphi_2\,da_f=-\int_\Sg\escpr{\nabla_\Sg\varphi_1,\nabla_\Sg\varphi_2}da_f-\int_{\ptl\Sg}\varphi_1\,\frac{\ptl\varphi_2}{\ptl\nu}\,dl_f,
\]
where $\nu$ is the inner unit normal along $\ptl\Sg$ in $\Sg$.
\end{corollary}

\section{Homogeneous densities in Euclidean solid cones}
\label{sec:homogeneous}
\setcounter{equation}{0}

In this section we introduce and study some densities defined on solid cones of $\rrn$ whose associated weighted volume and area have a nice behaviour with respect to the family of dilations centered at the origin.

By a (smooth) \emph{solid cone} in $\rrn$ we mean a cone
\[
M=0\cone\de:=\{t\,p\,; \ t\geq 0, \ p\in\de\},
\]
where $\de$ is a smooth region (the union of a connected open set together with its $C^\infty$ boundary) of the unit sphere $\sph^n$. Clearly $M$ is closed, connected, and invariant under the family of dilations $h_t(p):=tp$ defined for $t>0$ and $p\in\rrn$. We denote by $\ptl M$ and $\text{int}(M)$ the topological boundary and interior of $M$, respectively. We will use the notation $M^*$ for the punctured cone $M-\{0\}$. Note that $M$ coincides with a closed half-space of $\rrn$ when $\de$ is a hemisphere. In the case $\de=\sph^n$ we get $M=\rrn$.

Let $M$ be a solid cone and $f=e^\psi$ a (smooth) density on $M^*$. Recall that $V_f(\Om)$ is the weighted volume of a set 
$\Om\subseteq M$ as defined in \eqref{eq:volarea}. For a smooth hypersurface $\Sg$ in $M$ we denote by $A_f(\Sg)$ the weighted area $A_f(\Sg,\text{int}(M))$ given in \eqref{eq:volarea}. According to this definition $\Sg\cap\ptl M$ does not contribute to $A_f(\Sg)$. We are interested in densities for which the functionals $V_f$ and $A_f$ are homogeneous with respect to dilations centered at $0$.

\begin{lemma}
\label{lem:vap}
Let $M\subeq\rrn$ be a solid cone and $f=e^\psi$ a density function on $M^*$. For any $t>0$, let $h_t(p):=tp$ with $p\in\rrn$. The following statements are equivalent:
\begin{itemize}
\item[(i)] there is $k\in\rr$ such that $V_f(h_t(\Om))=t^{n+k+1}\,V_f(\Om)$ for any $t>0$ and any $\Om\subseteq M$,
\item[(ii)] $f$ is $k$-homogeneous, i.e., $f(h_t(p))=t^k\,f(p)$ for any $t>0$ and any $p\in M^*$.
\end{itemize} 
Moreover, in such a case, we also have $A_f(h_t(\Sg))=t^{n+k}\,A_f(\Sg)$ for any smooth hypersurface $\Sg$ in $M$.
\end{lemma}

\begin{proof}
For any $t>0$ the Jacobian determinant of the dilation $h_t:\rrn\to\rrn$ equals $t^{n+1}$. Given a set $\Om\subeq M$ we get, by the change of variables formula, that
\begin{equation}
\label{eq:oso}
V_f(h_t(\Om))=\int_\Om t^{n+1}\,f(h_t(p))\,dv.
\end{equation}
If statement (i) holds, then we have 
\[
\int_\Om t^{n+1}\,f(h_t(p))\,dv=\int_\Om t^{n+k+1}\,f(p)\,dv,
\]
for any set $\Om\subeq M$ and any $t>0$. This equality implies (ii) since $f$ is continuous on $M^*$. On the other hand, statement (i) follows from (ii) by using \eqref{eq:oso}. Finally, the behaviour of $A_f$ with respect to $h_t$ comes again from the change of variables formula since the Jacobian determinant of the diffeomorphism $h_t:\Sg\to h_t(\Sg)$ equals $t^n$.
\end{proof}

The previous result leads us to the next definition. Let $M$ be a solid cone of $\rrn$ and $k\in\rr$. By a (smooth) \emph{homogenous density of degree} $k$ on $M$ we mean a $C^\infty$ positive function $f=e^\psi$ on $M^*$ whose restriction to any open segment leaving from $0$ is a monomial of degree $k$, i.e., $f(tp)=t^kf(p)$ for any $t>0$ and any $p\in M^*$. Sometimes we will simply say that $f$ is a \emph{$k$-homogeneous density}. Such a density is uniquely determined by its values on $\de=M\cap\sph^n$. In fact, we have
\begin{equation}
\label{eq:caract}
f(p)=f\left(\frac{p}{|p|}\right)\,|p|^k=(\eta\circ\pi)(p)\,|p|^k, \quad p\in M^*,
\end{equation}
where $\eta$ is the restriction of $f$ to $\de$ and $\pi:M^*\to\de$ is the retraction $\pi(p):=p/|p|$. 

\begin{remark}
\label{re:singularity}
The continuity of $f$ gives the existence of constants $A,B>0$ such that $A\leq(\eta\circ\pi)(p)\leq B$ for any $p\in M^*$. Therefore, for $k\neq 0$, the density $f$ has well-defined limits when $|p|\to 0$ and $|p|\to+\infty$. More precisely
\begin{equation}
\label{eq:limits}
\lim_{|p|\to 0}f(p)=
\begin{cases}
0, & \text{ if } k>0
\\
+\infty, & \text{ if } k<0
\end{cases}
\quad \text{ and }\quad
\lim_{|p|\to+\infty}f(p)=
\begin{cases}
+\infty, & \text{ if } k>0
\\
0, & \text{ if } k<0
\end{cases}.
\end{equation}  
Hence the origin may be understood as a singularity, in the sense that $f$ cannot be continuously extended to $M$ as a finite positive density. In the case $k=0$ the limits in \eqref{eq:limits} do not exist unless $f$ is a constant function. 
\end{remark}

\begin{examples}
1. A homogeneous density of degree $k=0$ is determined by a smooth function which is constant and positive along any open segment of the cone starting from $0$. As for the constant density $f=1$, weighted volume and area are homogeneous of degree $n+1$ and $n$, respectively, with respect to dilations centered at $0$.

2. Let $f$ be a linear function or a quadratic form on $\rrn$. Then, the restriction of $f$ to any solid cone where $f>0$ provides a homogeneous density of degree $1$ or $2$, respectively.

3. The radial densities $f(p)=|p|^k$ are homogenous of degree $k$. In fact, it follows easily from \eqref{eq:caract} that a $k$-homogeneous density $f$ is radial if and only if $f(p)=c\,|p|^k$ for some constant $c>0$.

4. Formula \eqref{eq:caract} shows that the family of homogeneous densities is extremely large. In fact, if $\eta:\de\to\rr$ is any smooth positive function on a smooth region $\de\subeq\sph^n$, then $f(p):=\eta(p/|p|)\,|p|^k$ defines a $k$-homogeneous density on the cone $M=0\cone\de$.
\end{examples}

As a consequence of the singularity at the origin, the weighted volume and area of a compact hypersurface containing $0$ may be infinite for a homogeneous density of degree $k<0$. This is illustrated in the next examples.

\begin{examples}
\label{ex:infinite}
1. Let $\Sg$ be a round sphere in $\rrn$  endowed with a homogeneous density of degree $k\leq-(n+1)$. By using polar coordinates it is easy to check that any conical sector $\{t\,p\,; \, 0\leq t\leq a, \, p\in R\}$ over a region $R\subeq\sph^n$ has infinite weighted volume. If $0\in\Sg$ then any connected component of $\rrn-\Sg$ contains at least one of these sectors, so that both components have infinite volume.

2. Any compact embedded hypersurface $\Sg$ containing $0$ in $\rrn$ with a homogeneous density of degree $k\leq-n$ satisfies $A_f(\Sg)=+\infty$. To see this we consider the function $m(r):=A(\Sg\cap B_r)/A(\ptl B_r)$, where $A(\cdot)$ stands for the Euclidean area and $B_r$ denotes the open ball of radius $r$ centered at $0$. It is well known that $\lim_{r\to 0^+}m(r)=1$. By taking into account \eqref{eq:caract}, we can find  constants $c, c'>0$ such that
\[
A_f(\Sg\cap B_r)\geq c\,r^k A(\Sg\cap B_r)=c'\,r^{n+k}\,m(r).
\]
We deduce that, if $n+k\leq 0$, then $A_f(\Sg\cap B_r)$ does not approach $0$ when $r\to 0$. As a consequence $A_f(\Sg)=+\infty$.
\end{examples}

In the following result we gather some analytical properties of homogeneous densities that will be used throughout the paper.  

\begin{lemma}
\label{lem:prop}
For a $k$-homogeneous density $f=e^\psi$ on a solid cone $M\subeq\rrn$ the follo\-wing equalities hold:
\begin{itemize}
\item[(i)] $(\nabla\psi)(tp)=t^{-1}\,(\nabla\psi)(p)$, for any $t>0$ and any $p\in M^*$,
\item[(ii)] $\escpr{(\nabla\psi)(p),p}=k$, for any $p\in M^*$,
\item[(iii)] $(\nabla^2\psi)_p(p,p)=-k$, for any $p\in M^*$.
\end{itemize}
\end{lemma}

\begin{proof}
The $k$-homogeneity of $f$ implies that $\psi(tp)=k\log(t)+\psi(p)$, for any $t>0$ and any $p\in M^*$.  Equalities (i) and (ii) can be obtained by differentiating in the previous formula. Equality (iii) is easily deduced from (i) and (ii). 
\end{proof}

The equalities in the previous lemma also follow from the calculus of the gradient and the Hessian of $\psi$. This  computation is contained in the next lemma, which allows us to characterize the nonnegativity of the $k$-dimensional Bakry-\'Emery-Ricci tensor defined in \eqref{eq:fkricci}.

\begin{lemma}
\label{lem:gradhess}
Let $M=0\cone\mathcal{D}$ be a solid cone over a smooth region $\mathcal{D}\subeq\sph^n$. Given a $k$-homogeneous density $f=e^\psi$ on $M$ and a point $p\in M^*$, the following equalities hold:
\begin{align*}
(\nabla\psi)(p)&=\frac{k}{|p|}\,\pi(p)+\frac{1}{|p|}\,(\nabla_{\sph^n}\mu)(\pi(p)),
\\
(\nabla^2\psi)_p(p,p)&=-k,
\\
(\nabla^2\psi)_p(p,v)&=\frac{-1}{|p|}\,\escpr{(\nabla_{\sph^n}\mu)(\pi(p)),v}, \ \emph{ for any } v\in T_{\pi(p)}\sph^n,
\\
(\nabla^2\psi)_p(u,v)&=\frac{k}{|p|^2}\escpr{u,v}+
\frac{1}{|p|^2}\,(\nabla^2_{\sph^n}\mu)_{\pi(p)}(u,v),\ \emph{ for any } u,v\in T_{\pi(p)}\sph^n,
\end{align*}
where $\mu=\psi_{|\mathcal{D}}$, the map $\pi:M^*\to\mathcal{D}$ is the retraction $\pi(p):=p/|p|$, and $\nabla_{\sph^n}\mu$, $\nabla^2_{\sph^n}\mu$ denote the gradient and the Hessian relative to $\sph^n$ of $\mu$.
\end{lemma}

\begin{proof}
Equation \eqref{eq:caract} implies that $\psi=\text{log}(f)=\psi_1+\psi_2$, where $\psi_1(p):=(\mu\circ\pi)(p)$ and $\psi_2(p):=k\log(|p|)$ for any $p\in M^*$. To prove the statement we compute the gradient and the Hessian of $\psi_i$ for any $i=1,2$. Fix a point $p\in M^*$. On the one hand, it is easy to check that
\[
(\nabla\psi_2)(p)=\frac{k}{|p|}\,\pi(p),\quad 
(\nabla^2\psi_2)_p(u,v)=\frac{k}{|p|^2}\escpr{u,v}-\frac{2k}{|p|^2}\escpr{u,\pi(p)}\escpr{v,\pi(p)},
\]
for any two vectors $u,v\in\rrn$. On the other hand, a straightforward computation gives
\[
(\nabla\psi_1)(p)=\frac{1}{|p|}\,(\nabla_{\sph^n}\mu)(\pi(p)),
\]
whereas
\begin{align*}
(\nabla^2\psi_1)_p(p,v)&=-\escpr{(\nabla\psi_1)(p),v}
=\frac{-1}{|p|}\escpr{(\nabla_{\sph^n}\mu)(\pi(p)),v}, \ \text{ for any } v\in T_{\pi(p)}\sph^n,
\\
(\nabla^2\psi_1)_p(u,v)&=\frac{1}{|p|^2}\,(\nabla^2_{\sph^n}\mu)_{\pi(p)}(u,v), \ \text{ for any } u,v\in T_{\pi(p)}\sph^n.
\end{align*}
Combining the previous equalities the proof follows.
\end{proof}

\begin{corollary}
\label{cor:ricgeq0}
Let $M=0\cone\mathcal{D}$ be a solid cone over a smooth region $\mathcal{D}\subeq\sph^n$. Consider a $k$-homogeneous density $f=e^\psi$ on $M$ and denote $\mu=\psi_{|\mathcal{D}}$. Then, the $k$-dimensional Bakry-\'Emery-Ricci tensor in \eqref{eq:fkricci} satisfies $\emph{Ric}^k_f\geq 0$ on $M^*$ if and only if 
\[
(\nabla^2_{\sph^n}\mu)_p(v,v)\leq\frac{-1}{k}\,\escpr{(\nabla_{\sph^n}\mu)(p),v}^2-k,
\]
for any $p\in\mathcal{D}$ and any unit vector $v\in T_p\sph^n$.
\end{corollary}

\begin{proof}
The necessary condition is an immediate consequence of the first and the fourth equalities in Lemma~\ref{lem:gradhess}. For the sufficient condition we write any vector $v\neq 0$ as the sum $v=u+w$, where $u$ is proportional to $p/|p|$ and $w$ is tangent to $\sph^n$ at $p/|p|$. The statement then follows by combining all the equalities in Lemma~\ref{lem:gradhess}. 
\end{proof}

\begin{remark}
\label{re:nobund}
The previous result implies that a $k$-homogeneous density defined in the whole space $\rrn$ cannot exist satisfying $k>0$ and $\text{Ric}^k_f\geq 0$. Otherwise, we would obtain $(\nabla^2_{\sph^n}\mu)_p<0$ for any $p\in\sph^n$, which contradicts that $\mu$ reaches its minimum. Obviously this argument does not hold for a cone $M$ with $\ptl M\neq\emptyset$, where homogeneous densities of degree $k>0$ and $\text{Ric}^k_f\geq 0$ can be found, see Examples~\ref{ex:models} and \ref{ex:monomial}. 
\end{remark}

Another characterization of the inequality $\text{Ric}^k_f\geq 0$ is given in the following lemma.

\begin{lemma}
\label{lem:matrimonio}
Let $M\subeq\rrn$ be a convex solid cone endowed with a homogeneous density $f=e^\psi$ of degree $k\neq 0$. If $k<0$ $($resp. $k>0$$)$ then the inequality $\emph{Ric}^k_f\geq 0$ on $M^*$ is equivalent to that $f^{1/k}$ is convex $($resp.~concave$)$ on $M^*$. 
\end{lemma}

\begin{proof}
A straightforward computation leads to the identity
\[
\nabla^2f^{1/k}=\frac{-1}{k}\,f^{1/k}\,\text{Ric}_f^k,
\]
from which the claim follows.
\end{proof}

\begin{example}
\label{ex:models}
From the previous lemma we deduce that a homogeneous density $f$ of degree $k\neq 0$ on a convex solid cone $M$ satisfies $\text{Ric}^k_f=0$ on $M^*$ if and only if $f=\xi^k$ for some positive linear function $\xi$ on $M$. 
\end{example}

We finish this section with a generalization for homogeneous densities in solid cones of some important equalities in $\rrn$ involving the divergence of the conformal vector field $X(p):=p$ and the behaviour of volume, area and mean curvature with respect to dilations. 

\begin{proposition}
\label{prop:keyfacts}
Let $f=e^\psi$ be a $k$-homogeneous density on a solid cone $M\subeq\rrn$. We consider the position vector field $X(p):=p$ and its associated one-parameter group of dilations $\phi_t(p):=e^tp$. Then we have: 
\begin{itemize}
\item[(i)] $\divv_f X=n+k+1$ in $M^*$,
\item[(ii)] $\divv_{\Sg,f}X=n+k$ along any smooth hypersurface $\Sg$ in $M^*$,
\item[(iii)] $V_f(\phi_t(\Om))=(e^t)^{n+k+1}\,V_f(\Om)$, for any $\Om\subeq M$,
\item[(iv)] $A_f(\phi_t(\Sg))=(e^t)^{n+k}\,A_f(\Sg)$, for any smooth hypersurface $\Sg$ in $M$,
\item[(v)] if $\Sg$ is a smooth hypersurface in $M^*$ with unit normal vector $N$, then 
\[
(H_f)_t(\phi_t(p))=e^{-t}\,H_f(p), \ \emph{ for any } p\in\Sg,
\]
where $(H_f)_t$ is the $f$-mean curvature in \eqref{eq:fmc} of the hypersurface $\phi_t(\Sg)$ with res\-pect to the unit normal $N_t$ such that $N_t(\phi_t(p))=N(p)$.
\end{itemize}
\end{proposition}

\begin{proof}
Equalities (i) and (ii) are consequences of \eqref{eq:divvol}, \eqref{eq:divsup} and Lemma~\ref{lem:prop} (ii). Statements (iii) and (iv) follow from Lemma~\ref{lem:vap}. So we only have to prove (v). Let $H_t$ be the Euclidean mean curvature of $\phi_t(\Sg)$ with respect to $N_t$. It is well known that $H_t(\phi_t(p))=e^{-t}H(p)$ for any $p\in\Sg$. By using Lemma~\ref{lem:prop} (i), we get 
\begin{align*}
(H_f)_t\big(\phi_t(p)\big)&=
\big(nH_t-\escpr{\nabla\psi,N_t}\big)\big(\phi_t(p)\big)
\\
&=n\,e^{-t}\,H(p)-\escpr{(\nabla\psi)(e^tp),N(p)}
\\
&=e^{-t}\,\big(nH-\escpr{\nabla\psi,N}\big)(p)=e^{-t}\,H_f(p).
\end{align*}
This completes the proof.
\end{proof}

\section{Variational formulae. Stationary and stable hypersurfaces}
\label{sec:variational}
\setcounter{equation}{0}

In this section we gather some variational properties of those hypersurfaces in a solid cone which are first and second order minima of the weighted area functional under a constraint on the separated weighted volume.

Let $M\subeq\rrn$ be a solid cone endowed with a $k$-homogeneous density $f=e^\psi$. We denote by $\Sg$ a smooth, compact, orientable hypersurface immersed in $M$. If $\Sg$ has non-empty boundary then we assume $\ptl\Sg\sub\ptl M$. This does not exclude that $\Sg$ and $\ptl M$ may be tangent at some interior point of $\Sg$.
If $\ptl\Sg=\emptyset$ then we adopt the convention that all the integrals along $\ptl\Sg$ vanish. When $\Sg$ is embedded and separates a bounded open set $\Om$ with closure contained in the punctured cone $M^*$, then we can apply the divergence theorem in Lemma~\ref{lem:divthvol} to the position vector field $X(p):=p$. By Proposition~\ref{prop:keyfacts} (i) and the fact that $X$ is tangent to $\ptl M$, we get
\[
(n+k+1)\,V_f(\Om)=\int_\Om\divv_f X\,dv_f=-\int_\Sg\escpr{X,N}da_f,
\]
where $N$ is the inner unit normal along $\Sg$. As in \cite[Eq.~(2.2)]{bdc} we can use the previous formula to define the \emph{oriented weighted volume} of an immersed hypersurface $\Sg$ by
\begin{equation}
\label{eq:orivol}
V_f(\Sg):=\frac{-1}{n+k+1}\,\int_\Sg\escpr{X,N}da_f,
\end{equation}
where $N$ is a fixed unit normal vector along $\Sg$ and $k\neq-(n+1)$. By using the change of variables formula and the fact that the tangent hyperplane to $\Sg$ remains invariant under dilations, we can prove that
\begin{equation}
\label{eq:voldil}
V_f(h_t(\Sg))=t^{n+k+1}\,V_f(\Sg),
\end{equation}
where $h_t(p)=tp$ for any $p\in\rrn$ and any $t>0$.

\begin{remark}
\label{re:k>0}
For $k>0$, a $k$-homogeneous density $f$ is bounded near $0$ by \eqref{eq:limits}. Hence, the right side integral in \eqref{eq:orivol} is finite even if $0\in\Sg$ and so, the volume $V_f(\Sg)$ is well defined for any compact and orientable hypersurface $\Sg$ immersed in $M$. For the same reason, the weighted area $A_f(\Sg)$ is finite when $k>0$ even if $0\in\Sg$. In Examples~\ref{ex:infinite} we found that, for $k<0$, the weighted area and volume of a compact hypersurface containing $0$ may be infinite. By this reason we shall always consider $\Sg\sub M^*$ when $k<0$. 
\end{remark}

By a \emph{variation} of $\Sg$ we mean a smooth family of hypersurfaces $\Sg_t$ with $t\in (-\eps,\eps)$ given by smooth immersions $\phi_t:\Sg\to M$ such that $\Sg_0=\Sg$ and $\ptl\Sg_t\sub\ptl M$ for any $t$. The associated \emph{velocity vector} is the vector field along $\Sg$ defined by $X_p:=(d/dt)|_{t=0}\,\phi_t(p)$. Note that $X$ is tangent to $\ptl M$ in the points of $\ptl\Sg-\{0\}$.  Let $V_f(t):=V_f(\Sg_t)$ and $A_f(t):=A_f(\Sg_t)$ be the corresponding volume and area functionals. The variation is said to be \emph{volume-preserving} if $V_f(t)=V_f(0)$ for any $t\in (-\eps,\eps)$. 

In order to compute the first derivative of $V_f(t)$ and $A_f(t)$ we must take care of the fact that, if the density has degree $k\neq 0$, then $|\psi|=|\log(f)|$ tends to $+\infty$ by \eqref{eq:limits} when we approach $0$. However, if the variation $\Sg_t$ leaves invariant a small neighborhood of $0$ in $\Sg$, then we can compute $A_f'(0)$ and $V_f'(0)$ as in \cite[Lem.~3.1]{rcbm}, \cite{kathe}, to get
\begin{equation}
\label{eq:1st}
A_f'(0)=-\int_\Sg H_f\,u\,da_f-\int_{\ptl\Sg}\escpr{X,\nu}dl_f,
\qquad V_f'(0)=-\int_\Sg u\,da_f,
\end{equation}
where $H_f$ is the $f$-mean curvature defined in \eqref{eq:fmc}, $u$ is the normal component of $X$, and $\nu$ is the inner unit normal along $\ptl\Sg$ in $\Sg$.

We say that the hypersurface $\Sg$ is \emph{$f$-stationary} if $A_f'(0)=0$ for any volume-preserving variation. If $A_f'(0)=0$ \emph{for any variation} then we will say that $\Sg$ is \emph{strongly $f$-stationary}.
By using the formulae in \eqref{eq:1st} we can prove the following result, see \cite{kathe}.

\begin{lemma}
\label{lem:stationary}
If $\Sg$ is $f$-stationary $($resp. strongly $f$-stationary$)$, then we have:
\begin{itemize}
\item[(i)] $H_f$ is constant along $\Sg-\{0\}$ $($resp. $H_f=0$ along $\Sg-\{0\}$$)$,
\item[(ii)] $\Sg$ meets $\ptl M$ orthogonally in the points of 
$\ptl\Sg-\{0\}$, 
\item[(iii)] $(A_f-H_f\,V_f)'(0)=0$ for any variation of $\Sg$ supported away from $0$.
\end{itemize}
Moreover, if $0\notin\Sg$ then statements \emph{(i)} and \emph{(ii)} imply that $\Sg$ is $f$-stationary $($resp. strongly $f$-stationary$)$.
\end{lemma}

\begin{example}[Spheres centered at $0$]
\label{ex:sphere1}
Let $\Sg$ be the intersection with the cone $M$ of a round sphere of radius $r$ centered at $0$. We consider the unit normal along $\Sg$ given by $N(p)=-p/r$. From \eqref{eq:fmc} and Lemma~\ref{lem:prop} (ii) we get $$H_f(p)=\frac{n+k}{r},$$ for any $p\in\Sg$. Clearly $\Sg$ meets orthogonally $\ptl M$ in the points of $\ptl\Sg$. It follows from Lemma~\ref{lem:stationary} that $\Sg$ is always $f$-stationary, and strongly $f$-stationary if and only if $k=-n$. 
\end{example}

\begin{example}[Spheres containing $0$]
\label{ex:spheres11}
Consider the radial homogeneous density $f(p)=|p|^k$ in $\rrn$. Let $\Sg$ be a round sphere of radius $r$ and center $p_0$ such that $0\in\Sg$. Take the unit normal $N(p)=(p_0-p)/r$. Note that $(\nabla\psi)(p)=kp/|p|^2$ for any $p\neq 0$, and so
\[
H_f(p)=\frac{n}{r}-\frac{k\escpr{p,N(p)}}{|p|^2}.
\] 
Let $p\in\Sg$ with $p\neq 0$. Putting $p=p_0-rN(p)$ and taking into account that $|p_0|=r$, we get $\escpr{p,N(p)}=\escpr{p_0,N(p)}-r$ and $|p|^2=-2r\,(\escpr{p_0,N(p)}-r)$. It follows that 
\[
H_f(p)=\frac{2n+k}{2r},
\]
for any $p\in\Sg-\{0\}$. From \eqref{eq:1st} we conclude that $\Sg$ is a critical point of $A_f$ for any volume-preserving variation supported away from $0$.
\end{example}

Suppose now that $\Sg$ is an $f$-stationary hypersurface of constant $f$-mean curvature $H_f$ away from the origin. For a variation fixing a small neighborhood of $0$ in $\Sg$, we can use computations similar to those in \cite[Prop.~3.6]{rcbm} and \cite{kathe} to obtain the second variation formula
\begin{equation}
\label{eq:2nd}
(A_f-H_f\,V_f)''(0)=\mathcal{I}_f(u,u).
\end{equation}
In the above equation $\mathcal{I}_f$ denotes the \emph{$f$-index form} of $\Sg$, i.e., the quadratic form on $C_0^\infty(\Sg-\{0\})$ defined by
\begin{equation}
\label{eq:index1}
\mathcal{I}_f(u,v):=\int_\Sg \left\{\escpr{\nabla_\Sg u,\nabla_\Sg v}-\left(\text{Ric}_f(N,N)+|\sg|^2\right)uv\right\} da_f-\int_{\ptl\Sg}\text{II}(N,N)\,uv\,dl_f,
\end{equation}
where $\text{Ric}_f$ is the $f$-Ricci tensor in \eqref{eq:fricci}, $|\sg|^2$ is the squared sum of the principal curvatures of $\Sg$, and $\text{II}$ is the Euclidean second fundamental form of $\ptl M-\{0\}$ with respect to the inner unit normal. From the integration by parts formula in Corollary~\ref{cor:ibp}, we get
\begin{equation}
\label{eq:i1i2}
\mathcal{I}_f(u,v)=\mathcal{Q}_f(u,v), \ \text{ for any } u,v\in C^\infty_0(\Sg-\{0\}),
\end{equation}
where 
\begin{align}
\label{eq:index2}
\mathcal{Q}_f(u,v):=&-\int_\Sg u\left\{\Delta_{\Sg,f}\,v+\left(\text{Ric}_f(N,N)+|\sg|^2\right) v\right\} da_f
\\
\notag
&-\int_{\ptl\Sg}u\,\left\{\frac{\ptl v}{\ptl\nu}+\text{II}(N,N)\,v\right\}dl_f.
\end{align}
Here $\Delta_{\Sg,f}$ is the $f$-Laplacian relative to $\Sg$ introduced in \eqref{eq:fsigmalaplacian}. Following the terminology in \cite{bdc} we call the second order linear operator
\begin{equation}
\label{eq:jacobi}
\mathcal{L}_fv:=\Delta_{\Sg,f}\,v+\left(\text{Ric}_f(N,N)+|\sg|^2\right)v
\end{equation}
the \emph{$f$-Jacobi operator} of $\Sg$. It was shown in \cite[Proof of Prop.~3.6]{rcbm} that this operator coincides with the derivative of the $f$-mean curvature along the variation. More precisely
\begin{equation}
\label{eq:jacobi2}
(\mathcal{L}_f u)(p)=\frac{d}{dt}\bigg|_{t=0}(H_f)_t(\phi_t(p)), 
\ \text{ for any } p\in\Sg-\{0\},
\end{equation}
where $(H_f)_t$ denotes the $f$-mean curvature along the hypersurface $\Sg_t$ given by the immersion $\phi_t:\Sg\to M$. By using that $\mathcal{Q}_f$ is symmetric we obtain the equality
\begin{equation}
\label{eq:ulfv}
\int_\Sg\big(u\,\mathcal{L}_f v-v\,\mathcal{L}_f u\big)\,da_f=\int_{\ptl\Sg}\left\{v\,\frac{\ptl u}{\ptl\nu}-u\,\frac{\ptl v}{\ptl\nu}\right\}dl_f,
\end{equation} 
for any two functions $u,v\in C_0^\infty(\Sg-\{0\})$.

Let $\Sg$ be an $f$-stationary hypersurface of constant $f$-mean curvature $H_f$. We say that $\Sg$ is \emph{strongly $f$-stable} if we have $(A_f-H_f\,V_f)''(0)\geq 0$ for any variation of $\Sg$. We say that $\Sg$ is \emph{$f$-stable} if $A_f''(0)\geq 0$ for any volume-preserving variation. Obviously a strongly $f$-stable hypersurface is $f$-stable. For a strongly $f$-stationary hypersurface the notion of strong $f$-stability is analogue to the classical stability for minimal hypersurfaces with free boundary in a domain of $\rrn$.

By the arguments in \cite[Lem.~2.4]{bdc}, any function $u\in C^\infty_0(\Sg-\{0\})$ with $\int_\Sg u\,da_f=0$ is the normal component of a volume-preserving variation of $\Sg$ supported away from $0$. Thus, we can deduce the following result from \eqref{eq:2nd} and \eqref{eq:i1i2}.

\begin{lemma}
\label{lem:stable}
Let $\Sg$ be an $f$-stationary hypersurface with index form $\mathcal{Q}_f$ defined in \eqref{eq:index2}.
\begin{itemize} 
\item[(i)] If $\Sg$ is strongly $f$-stable then $\mathcal{Q}_f(u,u)\geq 0$ for any $u\in C^\infty_0(\Sg-\{0\})$.
\item[(ii)] If $\Sg$ is $f$-stable then $\mathcal{Q}_f(u,u)\geq 0$ for any $u\in C^\infty_0(\Sg-\{0\})$ with $\int_\Sg u\,da_f=0$. 
\end{itemize}
Moreover, if $0\notin\Sg$ then the reverse statements also hold.
\end{lemma}

\begin{remark}
\label{re:pipote}
Note that, if the cone $M$ is convex and the $f$-Ricci tensor defined in \eqref{eq:fricci} satisfies $\text{Ric}_f\geq 0$, then two terms in the expression of $\mathcal{Q}_f(u,u)$ in \eqref{eq:index2} are nonpositive and so, the stability condition in the previous lemma becomes more restrictive. Hence, convexity of the cone and  nonnegativity of $\text{Ric}_f$ are natural hypotheses to prove classification results for $f$-stable hypesurfaces. 
\end{remark}

\begin{example}[Spheres centered at $0$]
\label{ex:sphere2}
Let $\Sg$ be the intersection with the cone $M$ of a round sphere of radius $r$ centered at $0$. We consider the unit normal $N(p)=-p/r$. Note that $|\sg|^2=n/r^2$ in $\Sg$ and $\text{II}(N,N)=0$ along $\ptl\Sg$ since the inner unit normal to $\ptl M-\{0\}$ is constant along any open segment leaving from $0$. On the other hand, by \eqref{eq:fricci} and Lemma~\ref{lem:prop} (iii), we get $\text{Ric}_f(N,N)=k/r^2$. All this together gives
\[
\mathcal{Q}_f(u,u)=\mathcal{I}_f(u,u)=\int_\Sg\big\{|\nabla_\Sg u|^2-\frac{n+k}{r^2}\,u^2\big\}\,da_f.
\]
It follows from Lemma~\ref{lem:stable} that $\Sg$ is strongly $f$-stable if and only if $k\leq -n$. In the particular case of the homogeneous radial density $f(p)=|p|^k$ in $\rrn$, we deduce from \cite[Thm.~3.10]{rcbm} that $\Sg$ is $f$-stable if and only if $k\leq 0$. If $k>0$, $M$ is convex, and the $k$-dimensional Bakry-\'Emery-Ricci tensor in \eqref{eq:fkricci} satisfies $\text{Ric}_f^k\geq 0$, then any spherical cap $\Sg$ is an $f$-isoperimetric solution, i.e., minimizes the weighted area among all hypersurfaces in the cone separating a fixed weighted volume, as it is proved in \cite{cabre3} and \cite{homoiso}. In particular, $\Sg$ is $f$-stable, see also \cite[Re.~1.5]{cabre3}.
\end{example}

\section{Classification of compact $f$-stable hypersurfaces}
\label{sec:main}
\setcounter{equation}{0}

In this section we provide sufficient conditions on a solid cone with a homogeneous density $f$ to ensure that the unique compact $f$-stable hypersurfaces are intersections with the cone of round spheres centered at the vertex. In order to get some information from the stability condition we must use the inequality in Lemma~\ref{lem:stable} (ii) with a suitable test function. This will be done by following the arguments employed in \cite{bdc} to show that the round spheres are the unique compact, orientable, stable hypersurfaces immersed in $\rrn$ with the constant density $f=1$. 

We first prove an integral formula, which gives us the test function that will be inserted in the index form \eqref{eq:index2} to establish our main result.

\begin{proposition}[Minkowski formula]
\label{prop:mink}
Let $M\subeq\rrn$ be a solid cone endowed with a $k$-homogeneous density $f=e^\psi$. Consider a smooth, compact, orientable hypersurface $\Sg$ immersed in $M^*$ with $\ptl\Sg\sub\ptl M$ and such that $\Sg$ meets $\ptl M$ orthogonally in the points of $\ptl\Sg$. Let $N$ be a unit normal along the hypersurface, $H_f$ the $f$-mean curvature defined in \eqref{eq:fmc}, and $X(p):=p$ the position vector field in $\rrn$. Then we have
\begin{equation}
\label{eq:mink1}
\int_\Sg (n+k+H_f\escpr{X,N})\,da_f=0.
\end{equation}
If, in addition, $k\neq -(n+1)$ and $H_f$ is constant along $\Sg$, then
\begin{equation}
\label{eq:mink2}
(n+k)\,A_f(\Sg)=(n+k+1)\,H_f\,V_f(\Sg),
\end{equation}
where $A_f(\Sg)$ is the weighted area and $V_f(\Sg)$ in the oriented weighted volume. 
\end{proposition}

\begin{proof}
We apply the divergence theorem in Lemma~\ref{lem:divthsup} together with equality $\divv_{\Sg,f}X=n+k$ in Proposition~\ref{prop:keyfacts} (ii). We get
\[
\int_\Sg (n+k)\,da_f=-\int_\Sg H_f\escpr{X,N}da_f-\int_{\ptl\Sg}\escpr{X,\nu}dl_f,
\]
where $\nu$ is the inner unit normal along $\ptl\Sg$ in $\Sg$. The orthogonality condition between $\Sg$ and $\ptl M$ implies that $\nu$ coincides with the inner unit normal to $\ptl M$ along $\ptl\Sg$. Hence, formula \eqref{eq:mink1} follows from the previous equality by using that $X$ is tangent to $\ptl M-\{0\}$. Equality \eqref{eq:mink2} is a consequence of \eqref{eq:mink1} and the definition of $V_f(\Sg)$ in \eqref{eq:orivol}.
\end{proof}

\begin{remark}
Another proof of \eqref{eq:mink1} follows by applying the divergence theorem in Lemma~\ref{lem:divthsup} to the vector field $X^\top:=X-\escpr{X,N}N$ with $X(p)=p$, and
taking into account that $\text{div}_{\Sg,f}X^\top=n+k+H_f\escpr{X,N}$.
\end{remark}

\begin{remarks}
\label{re:minkowski}
1. Suppose that $\Sg$ is embedded, $f$-stationary, and separates a bounded open set $\Om$ with $\overline{\Om}\sub M^*$. In this case, the weighted volume $V_f(\Om)$ is well defined even in the case $k=-(n+1)$. By taking the inner unit normal $N$ along $\Sg$, formula \eqref{eq:mink2} reads
\[
(n+k)\,A_f(\Sg)=(n+k+1)\,H_f\,V_f(\Om).
\]
Thus, a compact hypersurface in the previous conditions cannot exist if $k=-(n+1)$. 

2. As a consequence of \eqref{eq:mink2} it follows that $H_f\neq 0$ provided $k\neq-n$.

3. In Example~\ref{ex:sphere1} we showed that, if $k=-n$, then the intersection $\Sg$ of $M$ with any round sphere centered at $0$ is strongly $f$-stationary. In fact, equation \eqref{eq:mink2} implies that any $f$-stationary hypersurface $\Sg$ immersed in $M^*$ satisfies $H_f=0$ when $k=-n$ and so, $\Sg$ is strongly $f$-stationary by Lemma~\ref{lem:stationary}. This property does not hold if $0\in\Sg$, as it is illustrated in Example~\ref{ex:spheres11}.

4. For the constant density $f=1$ in $\rrn$ we deduce the known Minkowski identity $A(\Sg)=(n+1)H\,V(\Om)$ relating Euclidean area, volume and mean curvature \cite{hsiung}.
\end{remarks}

For the case of the constant density $f=1$ in $\rrn$ formula $\eqref{eq:mink1}$ says that the function $u=1+H\escpr{X,N}$ has mean zero with respect to the Euclidean element of area. This is the test function used in \cite{bdc} and later interpreted in \cite{wente} as the normal component of the variation of $\Sg$ obtained by equidistant hypersurfaces dilated to keep the enclosed volume constant. In the next result we show that the function $u=n+k+H_f\escpr{X,N}$ in \eqref{eq:mink1} admits the same interpretation.

\begin{lemma}
\label{lem:inter}
Let $M\subeq\rrn$ be a solid cone endowed with a homogeneous density $f=e^\psi$ of degree $k\neq -(n+1)$ and $k\neq -n$. Consider a smooth, compact, orientable and $f$-stationary hypersurface $\Sg$ immersed in $M^*$ with $\ptl\Sg\sub\ptl M$. Let $N$ be a unit normal along $\Sg$ and define the variation $\varphi_t(p):=p+t(n+k)N(p)$. For any $t$ small enough, let $s(t)$ be the positive number such that $V_f(s(t)\,\varphi_t(\Sg))=V_f(\Sg)$. Then, the normal component of the variation 
\[
\phi_t(p):=s(t)\,\varphi_t(p)=s(t)\,(p+t(n+k)\,N(p))
\]
equals $n+k+H_f\escpr{X,N}$, where $H_f$ is the $f$-mean curvature defined in \eqref{eq:fmc}, and $X(p):=p$ is the position vector field in $\rrn$.
\end{lemma}

\begin{remark}
Note that, up to the constant factor $n+k$, the variation $\varphi_t$ provides the classical deformation of $\Sg$ by parallel hypersurfaces.
\end{remark}

\begin{proof}[Proof of Lemma~\ref{lem:inter}]
From Lemma~\ref{lem:stationary} we know that $H_f$ is constant and $\Sg$ meets $\ptl M$ orthogonally in the points of $\ptl\Sg$. Let $V_f(t):=V_f(\Sg_t)$, where $\Sg_t:=\varphi_t(\Sg)$. The velocity vector field of the variation $\varphi_t$ equals $(n+k)\,N$ along $\Sg$. By using \eqref{eq:1st} and \eqref{eq:mink2} we get
\[
V_f'(0)=-(n+k)\,A_f(\Sg)=-(n+k+1)\,H_f\,V_f(\Sg).
\] 
Thus $V_f'(0)\neq 0$ and we can find $\eps>0$ such that $V_f(t)$ has the same sign as $V_f(0)=V_f(\Sg)$ for any $t\in (-\eps,\eps)$. On the other hand, equation \eqref{eq:voldil} implies that
\[
V_f(\Sg)=V_f(s(t)\,\Sg_t)=s(t)^{n+k+1}\,V_f(t),
\] 
from which $s(t)=(V_f(\Sg)/V_f(t))^{1/(n+k+1)}$. Hence the function $s:(-\eps,\eps)\to\rr$ is smooth and positive with $s(0)=1$. By differentiating at $t=0$ in the previous equality, and taking into account the computation of $V_f'(0)$ above, we have
\[
0=(n+k+1)\,V_f(\Sg)\,(s'(0)-H_f),
\]
so that $s'(0)=H_f$. Therefore the velocity vector associated to $\phi_t(p)=s(t)\,\varphi_t(p)$ equals
\[
\frac{d}{dt}\bigg|_{t=0}\phi_t(p)=H_f\,X(p)+(n+k)\,N(p),
\]
and the claim follows.
\end{proof}

As pointed out in Remark~\ref{re:pipote}, in order to obtain classification results for $f$-stable hypersurfaces it is natural to assume convexity of the cone and a non-negative lower bound on the $f$-Ricci tensor defined in \eqref{eq:fricci}: in this way some terms in the index form \eqref{eq:index2} are nonpositive and the stability inequality in Lemma~\ref{lem:stable} (ii) becomes more restrictive. In fact, we are now ready to prove the following result for $f$-stable hypersurfaces in the punctured cone $M^*$.

\begin{theorem}
\label{th:main}
Let $M\subeq\rrn$ be a convex solid cone endowed with a  homogeneous density $f=e^\psi$ of degree $k\neq -(n+1)$. Suppose that $k<-n$ or $k>0$, and that the $k$-dimensional Bakry-\'Emery-Ricci tensor in \eqref{eq:fkricci} satisfies $\emph{Ric}_f^k\geq 0$. Let $\Sg$ be a smooth, compact, orientable hypersurface immersed in $M^*$ with $\ptl\Sg\sub\ptl M$. If $\Sg$ is $f$-stable, then $\Sg$ is the intersection with $M$ of a round sphere centered at the vertex.
\end{theorem}

\begin{remark}
The hypothesis $k\neq-(n+1)$ is quite natural. Indeed, it follows from \eqref{eq:orivol} and Remark~\ref{re:minkowski} that, for $k=-(n+1)$, there are neither immersed nor embedded compact $f$-stationary hypersurfaces contained in the punctured cone $M^*$.
\end{remark}

\begin{proof}[Proof of Theorem~\ref{th:main}]
Let $N$ be a unit normal vector along $\Sg$. From Lemma~\ref{lem:stationary} we know that the $f$-mean curvature $H_f$ of $\Sg$ is constant and $\Sg$ meets orthogonally $\ptl M$ in the points of $\ptl\Sg$. Let $X(p):=p$ be the position vector field in $\rrn$ and $g:=\escpr{X,N}$ the support function of $\Sg$. By \eqref{eq:mink1}, the function $u:=n+k+H_fg$ has mean zero with respect to the weighted element of area. As $\Sg$ is $f$-stable, we deduce from Lemma~\ref{lem:stable} (ii) that the $f$-index form introduced in \eqref{eq:index2} satisfies the inequality $\mathcal{Q}_f(u,u)\geq 0$, that is
\begin{equation}
\label{eq:main2}
0\leq -\int_\Sg u\,\mathcal{L}_f u\,da_f-\int_{\ptl\Sg}u\left\{\frac{\ptl u}{\ptl\nu}+\text{II}(N,N)\,u\right\}dl_f,
\end{equation}
where $\mathcal{L}_f$ is the $f$-Jacobi operator in \eqref{eq:jacobi}, $\nu$ is the inner unit normal along $\ptl\Sg$ in $\Sg$, and $\text{II}$ is the second fundamental form of $\ptl M-\{0\}$ with respect to the inner unit normal. 

Now, we compute the two integrals in \eqref{eq:main2}. By using the linearity of $\mathcal{L}_f$ we get
\begin{equation}
\label{eq:main3}
\int_\Sg u\,\mathcal{L}_f u\,da_f=(n+k)\int_\Sg u\,\mathcal{L}_f (1)\,da_f+H_f\int_\Sg u\,\mathcal{L}_f g\,da_f.
\end{equation}
Consider the one-parameter group $\phi_t(p):=e^tp$ of dilations associated to $X$, and the induced variation $\Sg_t:=\phi_t(\Sg)$ of $\Sg$. By equation \eqref{eq:jacobi2} and Proposition~\ref{prop:keyfacts} (v), we obtain 
\begin{equation}
\label{eq:lfg}
(\mathcal{L}_f g)(p)=\frac{d}{dt}\bigg|_{t=0}(H_f)_t(\phi_t(p))=
\frac{d}{dt}\bigg|_{t=0}e^{-t}\,H_f(p)=-H_f.
\end{equation}
Equalities \eqref{eq:lfg} and \eqref{eq:mink1} imply that the second integral at the right side of \eqref{eq:main3} vanishes. Hence, we have
\begin{align}
\label{eq:main4}
\int_\Sg u\,\mathcal{L}_f u\,da_f&=(n+k)^2\int_\Sg\mathcal{L}_f(1)\,da_f
+(n+k)\,H_f\int_\Sg g\,\mathcal{L}_f(1)\,da_f
\\
\notag
&=(n+k)^2\int_\Sg\mathcal{L}_f(1)\,da_f+(n+k)\,H_f\int_\Sg \mathcal{L}_fg\,da_f+(n+k)\,H_f\int_{\ptl\Sg}\frac{\ptl g}{\ptl\nu}\,dl_f
\\
\notag
&=(n+k)^2\int_\Sg\bigg\{\mathcal{L}_f(1)-\frac{H^2_f}{n+k}\bigg\}\,da_f-(n+k)\,H_f\int_{\ptl\Sg}\text{II}(N,N)\,g\,dl_f.
\end{align}
To get the second equality we have taken into account \eqref{eq:ulfv}. For the third one, we have used \eqref{eq:lfg} together with the identity $\ptl g/\ptl\nu=-\text{II}(N,N)\,g$, which was proved in \cite[Lem.~4.8]{cones}. Again from this identity, we deduce
\[
\frac{\ptl u}{\ptl\nu}+\text{II}(N,N)\,u=(n+k)\,\text{II}(N,N),
\]
so that the boundary integral in \eqref{eq:main2} equals
\begin{equation}
\label{eq:main5}
(n+k)^2\int_{\ptl\Sg}\text{II}(N,N)\,dl_f+(n+k)\,H_f\int_{\ptl\Sg}\text{II}(N,N)\,g\,dl_f.
\end{equation}

By substituting \eqref{eq:main4} and \eqref{eq:main5} into \eqref{eq:main2}, we conclude that
\begin{align*}
0&\leq -(n+k)^2\left[\int_\Sg\bigg\{\mathcal{L}_f(1)-\frac{H^2_f}{n+k}\bigg\}\,da_f+\int_{\ptl\Sg}\text{II}(N,N)\,dl_f
\right]
\\
&=-(n+k)^2\left[\int_\Sg\bigg\{\text{Ric}_f(N,N)+|\sg|^2-\frac{H^2_f}{n+k}\bigg\}\,da_f+\int_{\ptl\Sg}\text{II}(N,N)\,dl_f
\right],
\end{align*}
where $\text{Ric}_f$ is the $f$-Ricci tensor in \eqref{eq:fricci} and $|\sg|^2$ is the squared sum of the principal curvatures of $\Sg$. The proof finishes by using the convexity of $M$ and Lemma~\ref{lem:umbilical} below.
\end{proof}

\begin{lemma}
\label{lem:umbilical}
Let $M\subeq\rrn$ be a solid cone endowed with a $k$-homoge\-neous density $f=e^\psi$ with $k<-n$ or $k>0$, and such that $\emph{Ric}_f^k\geq 0$. If $\Sg$ is a smooth,  orientable hypersurface immersed in $M^*$, then we have
\begin{equation}
\label{eq:sign}
\emph{Ric}_f(N,N)+|\sg|^2-\frac{H^2_f}{n+k}\geq 0.
\end{equation}
Moreover, if $\Sg$ is compact with $\ptl\Sg\sub\ptl M$ and $f$-stationary, then equality holds in $\Sg$ if and only if $\Sg$ is the intersection with $M$ of a round sphere centered at the vertex.
\end{lemma}

\begin{proof}
Let $H$ be the Euclidean mean curvature of $\Sg$. Recall that $|\sg|^2\geq nH^2$ with equality in $\Sg$ for $n\geq 2$ if and only if $\Sg$ is totally umbilical \cite[Lem.~3.2]{bdc}. On the other hand, the hypothesis $\text{Ric}_f^k\geq 0$ is equivalent by \eqref{eq:fkricci} to $\text{Ric}_f\geq (1/k)(d\psi\otimes d\psi)$. By using the definition of $f$-mean curvature in \eqref{eq:fmc} and simplifying, we obtain
\begin{align*}
\text{Ric}_f(N,N)+|\sg|^2-\frac{H^2_f}{n+k}&\geq
\frac{\escpr{\nabla\psi,N}^2}{k}+nH^2-\frac{n^2H^2+\escpr{\nabla\psi,N}^2-2nH\escpr{\nabla\psi,N}}{n+k}
\\
&=\frac{n}{k\,(n+k)}\,\big(\escpr{\nabla\psi,N}+kH\big)^2\geq 0. 
\end{align*}
The fact that equality holds in \eqref{eq:sign} if $\Sg$ is contained inside a round sphere centered at the vertex follows from the computations in Examples~\ref{ex:sphere1} and \ref{ex:sphere2}. Conversely, let us suppose that we have equality in \eqref{eq:sign} for a compact $f$-stationary hypersurface $\Sg$ with $\ptl\Sg\sub\ptl M$. Then we get these identities along $\Sg$
\[
|\sg|^2=nH^2,\quad \text{Ric}_f(N,N)=\frac{\escpr{\nabla\psi,N}^2}{k},
\quad \escpr{\nabla\psi,N}=-kH.
\]
Note that $H_f=(n+k)H\neq0$ by \eqref{eq:fmc} and \eqref{eq:mink1}. From the equality $|\sg|^2=nH^2$ we conclude that there exists $r>0$ such that any connected component of $\Sg$ is contained in a round sphere of radius $r$. To finish the proof it suffices to show that the center of any of these spheres is the vertex of the cone.

Let $S$ be a connected component of $\Sg$ centered at $p_0$. We consider the unit normal along $S$ given by $N(p)=(p_0-p)/r$. By equality $\escpr{\nabla\psi,N}=-kH=-k/r$ and Lemma~\ref{lem:prop} (ii), we get $\escpr{(\nabla\psi)(p),p_0}=0$ for any $p\in S$. Let $q_0\in S$ be a point at maximum distance from the vertex. If $q_0\in\ptl S$ then the orthogonality condition along $\ptl S$ implies that $q_0$ is proportional to $N(q_0)$, so that $p_0\in\ptl M$ or $p_0\in -(\ptl M)$. Cutting with a plane passing through $0$ and containing $\{\nu(q_0),N(q_0)\}$ we see that $p_0\in\ptl M$. In the case $q_0\in S-\ptl S$ a similar argument gives $p_0\in\text{int}(M)$. Suppose $p_0\neq 0$. Then, we would find $t>0$ such that $q_0=tp_0$. By using equalities (i) and (ii) in Lemma~\ref{lem:prop}, we would get $\escpr{(\nabla\psi)(q_0),p_0}=t^{-1}\escpr{(\nabla\psi)(p_0),p_0}=k/t\neq 0$, a contradiction.
\end{proof}

\begin{remark}[Planar cones]
We must note that Theorem~\ref{th:main} holds in the planar case even if we do not assume convexity of the cone $M$. In this case the proof is simpler since $\text{II}(N,N)$ identically vanishes along $\ptl\Sg$.
\end{remark}

Now we will extend Theorem~\ref{th:main} to $f$-stable hypersurfaces \emph{that may contain the vertex} of a cone with a homogeneous density of positive degree. Note that this is a non-trivial task, since the vertex is a singular point for both, the cone and the density. However, for positive degree we know from \eqref{eq:limits} that the density is bounded near the vertex. This fact will allow us to perform an approximation argument similar to the one in \cite[Sect.~4]{cones}. The key ingredient is the existence of a sequence of functions as in the following lemma.

\begin{lemma}
\label{lem:aprox}
Let $\Sg$ be a smooth, compact, orientable hypersurface embedded in a solid cone $M\subeq\rrn$ with a homogeneous density $f=e^\psi$ of degree $k\geq 0$. We suppose $n\geq 2$ or $n+k>2$. Then, there is a sequence $\{\var_\eps\}_{\eps>0}$ of functions in $C_0^\infty(\Sg-\{0\})$ satisfying:
\begin{itemize}
\item[(i)] $0\leq\var_\eps\leq 1$, for any $\eps>0$,
\item[(ii)] $\lim_{\eps\to 0}\var_\eps(p)=1$, for any $p\in\Sg-\{0\}$,
\item[(iii)] $\lim_{\eps\to 0}\int_\Sg |\nabla_\Sg\var_\eps|^2\,da_f=0$.
\end{itemize}
\end{lemma}

\begin{proof}
For $n\geq 2$ the construction of a sequence $\{\var_\eps\}_{\eps>0}$ satisfying (i), (ii), and whose Euclidean energy tends to zero follows from \cite[Lem.~2.4]{sz} and \cite[Lem.~3.1]{morgan-rit} if we consider $0$ as an isolated singularity of $\Sg$. The hypothesis $k\geq 0$ implies by \eqref{eq:caract} and \eqref{eq:limits} that $f$ is bounded near the origin, so that such a sequence also satisfies (iii). 

Let us prove the case $n+k>2$. We denote by $B_r$ the open ball of radius $r>0$ centered at $0$. Take a smooth function $\var:\rrn\to [0,1]$ such that $\var=0$ in $B_{1/2}$ and $\var=1$ in $\rrn-B_1$. We define $\var_\eps(p):=\var(p/\eps)$. This is a sequence of smooth functions with $\var_\eps=0$ in $B_{\eps/2}$, $\var_\eps=1$ in $\rrn-B_\eps$ and $|\nabla\var_\eps|^2\leq c/\eps^2$ for some constant $c>0$. Clearly this sequence satisfies properties (i) and (ii). On the other hand, it is well-known that the function $m(\eps):=A(\Sg\cap B_\eps)/A(\ptl B_\eps)$ tends to $1$ when $\eps\to 0$. Here $A(\cdot)$ stands for the Euclidean area. By taking into account \eqref{eq:caract} and that $k\geq 0$, we can find constants $c', c''>0$ such that
\[
A_f(\Sg\cap B_\eps)\leq c'\,\eps^k A(\Sg\cap B_\eps)
=c''\,\eps^{n+k}\,m(\eps).
\]
Therefore, we deduce that
\[
\int_\Sg|\nabla_\Sg\var_\eps|^2\,da_f\leq\int_{\Sg\cap B_\eps}|\nabla\var_\eps|^2\,da_f\leq\frac{c}{\eps^2}\,A_f(\Sg\cap B_\eps)\leq c'''\,\eps^{n+k-2}\,m(\eps),
\]
which tends to zero when $\eps\to 0$ since $n+k>2$.
\end{proof}

We are now ready to prove the following result.

\begin{theorem}
\label{th:main2}
Let $M\subeq\rrn$ be a convex solid cone endowed with a  $k$-homogeneous density $f=e^\psi$ such that $k>0$, $n+k>2$, and $\emph{Ric}_f^k\geq 0$. Let $\Sg$ be a smooth, compact, orientable hypersurface embedded in $M$ with $\ptl\Sg\sub\ptl M$. If $\Sg$ is $f$-stable, then $\Sg$ is the intersection with $M$ of a round sphere centered at the vertex.
\end{theorem}

\begin{proof}
We will show that the proof of Theorem~\ref{th:main} can be formally carried out even in the case $0\in\Sg$. We will explain briefly the steps of the argument and omit the details. We will follow the same notation as in the proof of Theorem~\ref{th:main}.

\emph{Step 1.} We take a sequence $\{\var_\eps\}_{\eps>0}$ as in Lemma~\ref{lem:aprox}. Following \cite[Lem.~4.4]{cones} we can prove that the divergence theorem in Lemma~\ref{lem:divthsup} holds for any smooth vector field $X$ on $\Sg-\{0\}$ satisfying $|X|^2, \divv_{\Sg,f}X\in L^1(\Sg,da_f)$ and $\escpr{X,\nu}\in L^1(\ptl\Sg,dl_f)$. Here we denote by $L^1(\Sg,da_f)$ and $L^1(\ptl\Sg,dl_f)$ the spaces of integrable functions with respect to the weighted measures $da_f$ and $dl_f$. As a consequence, we can extend Corollary~\ref{cor:ibp} and formula \eqref{eq:ulfv} to any pair of bounded functions $u_1, u_2\in C^\infty(\Sg-\{0\})$ such that $|\nabla_\Sg u_i|^2, \Delta_{\Sg,f}u_i\in L^1(\Sg,da_f)$ and $\ptl u_i/\ptl\nu\in L^1(\ptl\Sg,dl_f)$, for $i=1,2$.

\emph{Step 2.} We reason as in \cite[Lem.~4.5]{cones} to prove that the stability inequality $\mathcal{Q}_f(u,u)\geq 0$ in Lemma~\ref{lem:stable} (ii) is valid for any bounded function $u\in C^\infty(\Sg-\{0\})$ with $\int_\Sg u\,da_f=0$ provided $|\nabla_\Sg u|^2, \Delta_{\Sg,f}u\in L^1(\Sg,da_f)$ and $\ptl u/\ptl\nu\in L^1(\ptl\Sg,dl_f)$. We first use approximation to see that $\mathcal{I}_f(u,u)\geq 0$. Here the convexity of the cone $M$ and the nonnegativity of $\text{Ric}_f$ are essential in order to apply Fatou's lemma. Second, we use the generalization of Corollary~\ref{cor:ibp} obtained in Step 1 to deduce $\mathcal{Q}_f(u,u)=\mathcal{I}_f(u,u)\geq 0$.

\emph{Step 3.} We can use the stability inequality of Step 2 as in \cite[Lem.~4.7]{cones} to get two integrability properties, namely $\text{Ric}_f(N,N)+|\sg|^2\in L^1(\Sg,da_f)$ and $\text{II}(N,N)\in L^1(\ptl\Sg,dl_f)$. These are non-trivial facts since the density $f$ and the cone $M$ have a singularity at $0$.

\emph{Step 4.} We define $u:=n+k+H_fg$, where $g=\escpr{X,N}$ and $X(p)=p$. We can apply to the vector field $X$ the divergence theorem obtained in Step 1  since $\Sg$ is compact and $f$ is bounded near the origin. By using that $\Sg$ is $f$-stationary we deduce, as in Proposition~\ref{prop:mink}, that $\int_\Sg u\,da_f=0$. Moreover, by equality~\eqref{eq:lfg} we have 
\[
\Delta_{\Sg,f}u=H_f\,\Delta_{\Sg,f}g=-H^2_f-H_f\,(\text{Ric}_f(N,N)+|\sg|^2)\,g.
\] 
By using Step 3, it follows that $u$ satisfies all the integrability conditions stated in Step 2 to ensure that $\mathcal{Q}_f(u,u)\geq 0$.

From this point we can reproduce the proof of Theorem~\ref{th:main} to get the claim.
\end{proof}

We finish this section by showing that our stability results in Theorems~\ref{th:main} and~\ref{th:main2} can be applied in some relevant examples. 

\begin{example}
\label{ex:models2}
Let $\xi:\rrn\to\rr$ be a linear function and $M$ a convex solid cone contained in the open half-space of $\rrn$ where $\xi>0$. For any $k\neq 0$ we saw in Example~\ref{ex:models} that the density $f:=\xi^k$ is $k$-homogeneous and $\text{Ric}_f^k$ vanishes identically. Thus we can apply our results provided $k<-n$ and $k\neq -(n+1)$, or $k>0$.
\end{example}

\begin{example}
\label{ex:monomial}
The following monomial densities were first studied for half-planes in \cite{maderna} and more generally in \cite{cabre2}. Let $\alpha_i$, $i=1,\ldots, n+1$, be non-negative real numbers such that $\alpha_i>0$ for some $i$. Consider $C:=\{(x_1,\ldots, x_{n+1})\in\rrn\,; x_i\geq 0\text{ for all } i \text{ such that }\alpha_i>0\}$, which is a convex cone. The function $f:C\to\rr$ given by $f(x_1,\ldots, x_{n+1}):=x_1^{\alpha_1}\cdots\,x_{n+1}^{\alpha_{n+1}}$ is a  homogeneous density of degree $k=\alpha_1+\ldots +\alpha_{n+1}$ vanishing along $\ptl C$. This density satisfies $\text{Ric}_f^k\geq 0$  by Lemma~\ref{lem:matrimonio} since $k>0$ and $f^{1/k}$ is concave. The main result in \cite{cabre2} implies that any spherical cap of $C$ centered at the vertex minimize weighted area among hypersurfaces separating the same weighted volume. Our stability results classify $f$-stable hypersurfaces inside any convex solid cone $M$ with $M^*\sub\text{int}(C)$. More interesting examples of positively homogeneous densities where the isoperimetric problem has been solved and our result is applied are described in \cite[Sect.~2]{cabre3}.
\end{example}

\begin{example}
Consider the $k$-homogeneous radial density $f(p)=|p|^k$ in a convex solid cone $M\subeq\rrn$. If $k<0$ then $\text{Ric}_f^k\geq 0$ by Corollary~\ref{cor:ricgeq0} since the restriction of $f$ to $M\cap\sph^n$ is constant. Hence Theorem~\ref{th:main} applies provided $k<-n$ and $k\neq -(n+1)$. In particular, in $\rrn$ with density $|p|^k$, 
$k<-n$, $k\neq -(n+1)$, the unique compact $f$-stable hypersurfaces which do not contain $0$ are round spheres centered at $0$. In fact, it was proved in \cite[Prop.~7.5]{dhht} that, if $k<-(n+1)$, then these spheres uniquely minimize the weighted area for fixed weighted volume. In the case $-(n+1)<k<0$ such spheres are $f$-stable by \cite[Thm.~3.10]{rcbm} but not $f$-isoperimetric, see \cite[Prop.~7.3]{dhht}. The isoperimetric property of round spheres centered at $0$ for arbitrary radial densities in $\rrn$ is an interesting question, which leads to the log-convex density conjecture, see \cite{lcdc}. 
\end{example}

\begin{example}
\label{ex:models3}
Let $M=0\cone\mathcal{D}$ be any convex solid cone in $\rrn$. Let us see that, besides the radial case $f(p)=|p|^k$, there are several homogeneous non-radial densities of degree $k<0$ where $\text{Ric}_f^k\geq 0$. Take any smooth function $\mu:\mathcal{D}\to\rr$ such that $c_\mu:=\max\{(\nabla^2_{\sph^n}\mu)_p(v,v);\,p\in\mathcal{D}, \, v\in T_p\sph^n, \, |v|=1\}$ is positive. Fix a number $k<0$ and consider the $k$-homogeneous density in $M$ defined by $f_\mu(p):=e^{\mu(p/|p|)}|p|^k$. Then, for any $a\in (0,-k/c_\mu]$, it is clear that 
$k+c_{a\mu}\leq 0$ and so, the density $f_{a\mu}$ has non-negative Bakry-\'Emery-Ricci tensor by Corollary~\ref{cor:ricgeq0}. Hence Theorem~\ref{th:main} applies provided $k<-n$ and $k\neq -(n+1)$.
\end{example}

\begin{remark}
\label{re:cabre}
As pointed out in Example~\ref{ex:sphere2}, the conditions of Theorem~\ref{th:main} imply that any round sphere centered at the vertex and intersected with the cone is $f$-stable. Indeed, for $k>0$ these spherical caps minimize the weighted area relative to the interior of the cone among hypersurfaces separating the same weighted volume. This isoperimetric property has been proved by Cabr\'e, Ros-Oton and Serra in \cite{cabre} and \cite{cabre3} for an arbitrary convex cone $M$ endowed with a continuous, homogeneous, non-negative density function $f$ of degree $k>0$, such that $f$ is positive and locally Lipschitz in the points of $\text{int}(M)$ and $f^{1/k}$ is concave in $\text{int}(M)$. In \cite{homoiso} we obtain uniqueness of the isoperimetric regions for smooth cones with $k$-homogeneous smooth densities satisfying $k>0$ and $\text{Ric}^k_f\geq 0$. From Lemma~\ref{lem:matrimonio} we deduce that the hypothesis $\text{Ric}^k_f\geq 0$ is equivalent, for $C^2$ homogeneous densities of positive degree, to the concavity of $f^{1/k}$ assumed in \cite{cabre} and \cite{cabre3}.
\end{remark}

\begin{remark}
Theorem~\ref{th:main} need not hold if the hypothesis $\text{Ric}_f^k\geq 0$ fails. Consider the density $f(p)=|p|^k$ with $k>0$ in $\rrn$. From the computations in Lemma~\ref{lem:gradhess} it is easy to check that $\text{Ric}_f^k(u,u)<0$ for any $u\neq 0$ tangent to $\sph^n$. It is known by \cite[Thm.~3.10]{rcbm} that round spheres about the origin are $f$-unstable. It was proved in \cite[Thm.~3.16]{dahlberg} that round circles \emph{containing the origin} uniquely minimize weighted length for fixed weighted area in the plane with density $f$. In higher dimension it is an open question if any $f$-isoperimetric or compact $f$-stable hypersurface is a sphere through the origin (such spheres have constant $f$-mean curvature by Example~\ref{ex:spheres11}). In \cite[Thm.~4.17]{dhht} we find examples of planar convex cones where compact $f$-stable curves are not contained into circles centered at $0$.
\end{remark}

\section{Classification of compact strongly $f$-stable hypersurfaces}
\label{sec:main2}
\setcounter{equation}{0}

In this last section we obtain some results for compact strongly $f$-stable hypersurfaces inside a solid cone $M$ with a homogeneous density $f$ of degree $k$. 

We first analyze the case $k>0$. Observe that, when the cone $M$ is convex and the $k$-dimensional Bakry-\'Emery-Ricci tensor satisfies $\text{Ric}^k_f\geq 0$, then we can apply Theorem~\ref{th:main} to deduce that a smooth, compact, orientable, strongly $f$-stable hypersurface $\Sg$ immersed in $M^*$ with $\ptl\Sg\sub\ptl M$ must be the intersection with $M$ of a round sphere centered at the vertex. However, these spherical caps are not strongly $f$-stable, as it is shown in Example~\ref{ex:sphere2}. In the next result we prove that the same happens if the $f$-Ricci tensor is nonnegative.

\begin{proposition}
\label{prop:nonexist}
Let $M\subeq\rrn$ be a convex solid cone endowed with a homogeneous density $f$ whose $f$-Ricci tensor satisfies $\emph{Ric}_f\geq 0$. Then, there are no smooth, compact, orientable, strongly $f$-stable hypersurfaces immersed in $M^*$ with boundary inside $\ptl M$.
\end{proposition}

\begin{proof}
We reason by contradiction. Suppose there was a hypersurface $\Sg$ in the conditions of the statement. By inserting $u=1$ in the index form \eqref{eq:index2} and using Lemma~\ref{lem:stable} (i), we would deduce
\[
0\leq\mathcal{Q}_f(u,u)=-\int_{\Sg}\big(\text{Ric}_f(N,N)+|\sg|^2\big)\,da_f-\int_{\ptl\Sg}\text{II}(N,N)\,dl_f.
\]
Hence the convexity of $M$ and the nonnegativity of $\text{Ric}_f$ would give us $|\sg|^2=0$ along $\Sg$. In particular, any connected component of $\Sg$ would be contained inside a hyperplane. Let $\Sg'$ be a component of $\Sg$. We take a point $p_0\in\Sg'$ at maximum distance from the vertex of the cone. The point $p_0$ can be at the interior or at the boundary of $\Sg'$. Anyway, the fact that $\Sg'$ meets $\ptl M$ orthogonally in the points of $\ptl\Sg'$ would imply that $p_0$ is a normal vector to $\Sg'$. As a consequence, $\Sg'$ would be contained inside the tangent hyperplane $P$ at $p_0$ to the round sphere $S$ of radius $|p_0|$ centered at $0$. This contradicts the definition of $p_0$, since $P$ is a support hyperplane for the closed ball bounded by $S$.
\end{proof}

Note that the hypothesis $\text{Ric}_f\geq 0$ implies that the density has degree $k\geq 0$ by the second equality in Lemma~\ref{lem:gradhess}. In particular, we can reproduce the approximation argument in the proof of Theorem~\ref{th:main2} to obtain the next result, which is valid for strongly $f$-stable hypersurfaces containing the vertex of the cone.

\begin{proposition}
\label{prop:nonexist2}
Let $M\subeq\rrn$ be a convex solid cone endowed with a $k$-homogeneous density $f$ such that $n\geq 2$ or $n+k>2$. If $\emph{Ric}_f\geq 0$ then there are no smooth, compact, orientable, strongly $f$-stable hypersurfaces embedded in $M$ with $\ptl\Sg\sub\ptl M$.
\end{proposition}

Now we focus on homogeneous densities with degree $k<0$.
The next corollary follows from Theorem~\ref{th:main} and Example~\ref{ex:sphere2} by taking into account that a strongly $f$-stable hypersurface is also $f$-stable.

\begin{corollary}
\label{cor:strongly}
Let $M\subeq\rrn$ be a convex solid cone endowed with a homogeneous density $f=e^\psi$ of degree $k\neq -(n+1)$. Suppose also that $k<-n$ and $\emph{Ric}_f^k\geq 0$. Let $\Sg$ be a smooth, compact, orientable hypersurface immersed in $M^*$ with $\ptl\Sg\sub\ptl M$. Then, $\Sg$ is strongly $f$-stable if and only if $\Sg$ is the intersection with $M$ of a round sphere centered at the vertex. 
\end{corollary}

It was pointed in Remark~\ref{re:minkowski} that any compact $f$-stationary hypersurface $\Sg$ immersed in a punctured solid cone $M^*$ with a homogeneous density $f$ of degree $k=-n$ satisfies $H_f=0$. In particular, $\Sg$ is strongly $f$-stationary by Lemma~\ref{lem:stationary} and the notion of strong $f$-stability is the density analogue to the classical stability for minimal hypersurfaces with free boundary in a domain of $\rrn$. On the other hand, we saw in Example~\ref{ex:sphere2} that the intersection with $M$ of any round sphere centered at the vertex is strongly $f$-stable for $k=-n$. In the next result we generalize Corollary~\ref{cor:strongly} by showing uniqueness of these spherical caps as strongly $f$-stable hypersurfaces.

\begin{theorem}
\label{th:minimal}
Let $M\subeq\rrn$, $n\geq 2$, be a convex solid cone endowed with a homogeneous density $f=e^\psi$ of degree $k=-n$ such that $\emph{Ric}_f^{-n}\geq 0$. Let $\Sg$ be a smooth, compact, connected, orientable hypersurface immersed in $M^*$ with $\ptl\Sg\sub\ptl M$. Then, $\Sg$ is strongly $f$-stable if and only $\Sg$ is the intersection with $M$ of a round sphere centered at the vertex.
\end{theorem}

\begin{proof}
Let $N$ be a unit normal vector along $\Sg$. We denote by $H_f$ and $H$ the $f$-mean curvature and the Euclidean mean curvature of $\Sg$, respectively. As $\Sg$ is $f$-stationary, we deduce from Lemma~\ref{lem:stationary} that $H_f$ is constant along $\Sg$, and that $\Sg$ meets orthogonally $\ptl M$ in the points of $\ptl\Sg$. By \eqref{eq:mink2} we get $H_f=0$ and so, $nH=\escpr{\nabla\psi,N}$ along $\Sg$. Since $\Sg$ is strongly $f$-stable we can apply Lemma~\ref{lem:stable} (i) to obtain $\mathcal{Q}_f(u,u)\geq 0$ for any $u\in C^\infty(\Sg)$, where $\mathcal{Q}_f$ is the $f$-index form defined in \eqref{eq:index2}. In particular, by taking $u=1$, we have
\[
0\leq-\int_\Sg\big(\text{Ric}_f(N,N)+|\sg|^2\big)\,da_f
-\int_{\ptl\Sg}\text{II}(N,N)\,dl_f.
\]
On the other hand, the hypothesis $\text{Ric}_f^{-n}\geq 0$ together with the inequality $|\sg|^2\geq nH^2$ and the fact that $nH=\escpr{\nabla\psi,N}$, gives us
\[
\text{Ric}_f(N,N)+|\sg|^2\geq\frac{-\escpr{\nabla\psi,N}^2}{n}+nH^2=0.
\]
This inequality and the convexity of the cone imply that
\[
|\sg|^2=nH^2,\quad\text{Ric}_f(N,N)=\frac{-\escpr{\nabla\psi,N}^2}{n},\quad\text{II}(N,N)=0,
\]
so that $\Sg$ is contained inside a hyperplane or a round sphere. By reasoning as in the proof of Proposition~\ref{prop:nonexist} we deduce that the first case is not possible.
In the second case we can reproduce the argument at the end of the proof of Lemma~\ref{lem:umbilical} to show that $\Sg$ is the intersection of $M$ with a round sphere centered at the vertex. 
\end{proof}

\begin{remarks}
1.  The previous result does not follow from Theorem~\ref{th:main}, where it was assumed that $k<-n$ or $k>0$. So, for the case $k=-n$, we cannot deduce that any $f$-stable hypersurface is the intersection with $M$ of a round sphere centered at $0$.

2. In the previous theorem we have to assume that $\Sg$ is connected to get a single sphere in the thesis. In fact, in the case $k=-n$, the union of finitely many spherical caps centered at $0$ turns out to be strongly $f$-stable by the computations in Example~\ref{ex:sphere2}. 
\end{remarks}

\begin{remark}[A note on the regularity hypotheses]
Theorems~\ref{th:main} and \ref{th:main2} hold for $C^3$ hypersurfaces in $C^2$ convex cones with $C^2$ densities. Theorem~\ref{th:minimal} is also valid for $C^2$ hypersurfaces. 
\end{remark}

\providecommand{\bysame}{\leavevmode\hbox to3em{\hrulefill}\thinspace}
\providecommand{\MR}{\relax\ifhmode\unskip\space\fi MR }
\providecommand{\MRhref}[2]{%
  \href{http://www.ams.org/mathscinet-getitem?mr=#1}{#2}
}
\providecommand{\href}[2]{#2}

\end{document}